\input ./style/arxiv-general.cfg
\documentclass[aop,MSNbibl,dvips]{arximspdf}
\makeatletter
   \@ifpackageloaded{graphicx}{}{\usepackage{graphicx}}
\makeatother
\usepackage{mathrsfs}
\usepackage{stfloats}

\doi{10.1214/15-AOP1034}
\volume{44}
\issue{4}
\pubyear{2016}
\firstpage{2770}
\lastpage{2816}
\docsubty{FLA}

\makeatletter
\renewcommand{\epsilon}{\varepsilon}
  \fnbelowfloat
\renewcommand{\mid}{|}
\renewcommand{\leqslant}{\leq}

\newcommand{\rrvert}{\vert}
\newcommand{\llvert}{\vert}

\newcommand{\Clconstcccpenki}{c_1}
\newcommand{\Clconstcccnulis}{c_2}
\newcommand{\Clconstcccseptyni}{c_3}
\newcommand{\Clconstcccdesimt}{c_4}
\newcommand{\Clconstcccvienas}{c_5}
\newcommand{\Clconstcccdu}{c_6}
\newcommand{\Clconstccavienas}{c_7}
\newcommand{\Clconstccadu}{c_8}
\newcommand{\Cdev}{c_9}
\newcommand{\Clconstccseptyn}{c_{10}}

\newcommand{\Clconstddcvienas}{d_1}
\newcommand{\Clconstddbv}{d_2}
\newcommand{\Clconstddcptr}{d_3}
\newcommand{\Clconstddbke}{d_4}
\newcommand{\Clconstddcppe}{d_5}
\newcommand{\Clconstddcppps}{d_6}
\newcommand{\deptyni}{d_7}
\def\E{\mathbb{E}}
\def\N{\mathbb{N}}
\def\P{\mathbb{P}}
\def\Z{\mathbb{Z}}
\def\R{\mathbb{R}}
\def\S{\mathbb{S}}
\def\A{\mathcal{A}}
\def\B{\mathcal{B}}
\def\G{\mathcal{G}}
\def\W{\mathcal{W}}
\def\K{\mathfrak{K}}
\def\NN{\mathcal{N}}
\def\NNN{\mathfrak{N}}
\def\MM{\mathcal{M}}
\def\rtun{{}^{(1)}{r_t}}
\def\rtdeux{{}^{(2)}{r_t}}
\newcommand{\un}{\mathbf{1}}
\def\DD{\mathcal{D}}
\def\F{\mathcal{F}}
\newtheorem{prop}{Proposition}
\newtheorem{theorem}{Theorem}
\newtheorem{coroll}{Corollary}
\newtheorem{lemma}{Lemma}
\makeatother

\begin{document}
\begin{frontmatter}

\title{Fluctuations of the front in a one-dimensional model for the
spread of an infection\thanksref{T1}}
\runtitle{Fluctuations of the front in a 1d infection model}

\begin{aug}
\author[A]{\fnms{Jean}~\snm{B\'erard}\corref{}\thanksref{T2}\ead[label=e1]{jberard@unistra.fr}}
\and
\author[B]{\fnms{Alejandro}~\snm{Ram\'irez}\thanksref{T3}\ead[label=e2]{aramirez@mat.puc.cl}}
\runauthor{J. B\'erard and A. Ram\'irez}
\affiliation{Universit\'e de Strasbourg and Pontificia Universidad
Cat\'
olica de Chile}
\address[A]{I.R.M.A. CNRS UMR 7501\\
Universit\'e de Strasbourg\\
7 rue Ren\'e Descartes\\
F-67084 Strasbourg Cedex\\
France\\
\printead{e1}}
\address[B]{Facultad de Matem\'aticas\\
Pontificia Universidad Cat\'olica de Chile\\
Vicu\~na Mackenna 4860\\
Macul, Santiago\\
Chile\\
\printead{e2}}
\end{aug}
\thankstext{T1}{Supported in part by ECOS-Conicyt Grant CO9EO5.}
\thankstext{T2}{Supported in part by ANR project MEMEMO2.}
\thankstext{T3}{Supported in part by Fondo Nacional de Desarrollo Cient\'\i fico y Tecnol\'ogico Grant 1100298 and by Iniciativa Cient\'ifica Milenio Grant NC130062.}

%
\received{\smonth{5} \syear{2014}}
%
\revised{\smonth{4} \syear{2015}}

%
\begin{abstract}
We study the following microscopic model of infection or epidemic
reaction: red and blue particles perform
independent nearest-neighbor continuous-time symmetric random walks on the
integer lattice $\mathbb{Z}$ with jump rates $D_R$ for red particles
and $D_B$ for blue particles, the interaction rule
being that blue particles turn red upon contact with a red particle. The
initial condition consists of i.i.d. Poisson particle numbers at each site,
with particles at the left of the origin being red, while particles at the
right of the origin are blue. We are interested in the dynamics of the front,
defined as the rightmost position of a red particle. For the case
$D_R=D_B$, Kesten and Sidoravicius
established that the front moves ballistically, and more precisely that it
satisfies a law of large numbers. Their proof is based on a multi-scale
renormalization technique, combined with approximate sub-additivity
arguments. In this paper,
we build a renewal structure for the front propagation process, and as
a corollary we obtain a central limit
theorem for the front when $D_R=D_B$.
Moreover, this result can be extended to the case where $D_R>D_B$, up to
modifying the dynamics so that blue particles turn red upon contact
with a
site that has previously been occupied by a red particle. Our approach
extends the renewal structure approach developed by
Comets, Quastel and Ram\'\i rez
for the so-called frog model, which corresponds to the $D_B=0$ case.
\end{abstract}

%
\end{frontmatter}

\setcounter{footnote}{3}
\section{Introduction}

Consider the following microscopic model of infection
or epidemic reaction
on the integer lattice $\mathbb Z$.
There are two types of particles: red and blue, both moving as
independent,
continuous-time, symmetric, nearest-neighbor random walks, with total
jump rate $D_R$ for red particles
and $D_B$ for blue particles.
The interaction rule between particles is the following: when a red
particle jumps to a site where there are
blue particles, all of them immediately become red particles; when a
blue particle jumps to a site
where there are red particles, it immediately becomes a
red particle. The initial condition is the following: at time zero,
each site in $\Z$ bears a random number of particles whose distribution
is Poisson with parameter $\rho>0$, the numbers of particles at
distinct sites being independent. Moreover, particles at the left of
the origin (including the origin) are red, while particles at the right
of the origin are blue.
We are interested in the asymptotic behavior of the rightmost site
$r_t$ occupied by a red particle at time $t$, which we call the \textit{front}.

Such particle systems have received attention in the
physical literature, as microscopic stochastic models which, in the
limit of
a large average number of particles per lattice site, yield reaction--diffusion
equations describing the \mbox{propagation} of a front, the prototypical
example being
the Fisher--Kolmogorov--Petrovsky--Piscounov equation; see, for
example, \cite{MaiSokKuzBlu,MaiSokBlu,MaiSokBlu2,KumTrip}. We refer to
\cite{Pan}
for an extensive review of the subject from a theoretical physics perspective.

On the other hand, according to
\cite{KesSid}, this model was suggested within the mathematics community
by Frank Spitzer around 1980, but rigorous mathematical results
describing the
behavior of the front have been difficult to obtain.

Indeed, the only two special cases for which ballisticity of the front
and a law of large numbers have been mathematically established are the
following:
\begin{itemize}
\item$D_R>D_B=0$; this is the so-called \textit{frog model} \cite
{RamSid,AlvMacPop}. Beyond ballisticity and the law of large numbers, a
central limit theorem and a large deviations principle have also been
obtained \cite{ComQuaRam,BerRam}.

\item$D_R=D_B>0$; this model will be referred to as the \textit{single-rate KS infection model}, after Kesten and Sidoravicius \cite
{KesSid,KesSid2}, where ``single rate'' emphasizes
the fact that red and blue particles share the same jump rate.
\end{itemize}

Specifically, in \cite{KesSid}, it is shown that the front moves
ballistically, in the sense that there exist two constants $C_1, C_2$
such that a.s.
%
\begin{equation}
\label{e:ballistique}0<C_2 \leq\liminf_{t \to
+\infty
}
t^{-1} r_t \leq\limsup_{t \to+\infty}
t^{-1} r_t \leq C_1 < +\infty.
\end{equation}
This result is strengthened in \cite{KesSid2} where it is shown that
there exists $0<v_*<+\infty$ such that a.s.,
%
\begin{equation}
\label{e:lgn} \lim_{t \to+\infty} t^{-1} r_t=v_*.
\end{equation}
Analogous results hold on $\Z^d$ for arbitrary $d \geq1$, with (\ref
{e:lgn}) being the one-dimensional version of a general shape theorem
proved in \cite{KesSid2}. Here, we are interested in
the fluctuations of $r_t$, and the first main result of the present
paper is the following.
%
\begin{theorem}
\label{t:tcl}
For the single-rate KS infection model, there exists a (nonrandom)
number $0<\sigma_*^2 < +\infty$ such that, as $\epsilon$ goes to zero,
\[
B_t^\epsilon:=\epsilon^{1/2} \bigl(r_{\epsilon^{-1}t}-
\epsilon^{-1} v_* t \bigr), \qquad t\ge0, %
\]
converges in law on the Skorohod space to a Brownian motion
with variance $\sigma_*^2$.
\end{theorem}

Note that the method used to derive the above results also yields the
convergence to an invariant distribution of the environment of
particles as seen from the
front.

For the general case in which $D_R \neq D_B$, an upper bound on the
speed similar to the one
in (\ref{e:ballistique}) is proved in \cite{KesSid}, but no
corresponding lower bound is available, so that even ballisticity as
described in (\ref{e:ballistique}) is an open question. We now
introduce a slight variation upon this model for which, when $D_R >
D_B>0$, it is indeed possible to derive results similar to those that
hold for the single-rate model. This variation consists in making the
infectious power of red particles \textit{remanent}, in the sense that a
blue particle turns red not only when it is in contact with a red
particle, but as soon as it is located at a site that has previously
been occupied by a red particle. We call this model the \textit{remanent
KS infection model}. In this context, it is natural to define the
position of the front at time $t$ as the rightmost position ever
occupied by a red particle up to time $t$.
We can then prove the two following results.

%
\begin{theorem}\label{t:lgn-remanent}
For the remanent KS infection model with $0<D_B \leq D_R$, there exists
$0<v_{\star}<+\infty$ such that a.s.,
\[
\lim_{t \to+\infty} t^{-1} r_t=v_{\star}.
\]
\end{theorem}

%
\begin{theorem}\label{t:tcl-remanent}
For the remanent KS infection model with $0<D_B \leq D_R$,
there exists a (nonrandom) number $0<\sigma_{\star}^2 < +\infty$ such
that, as $\epsilon$ goes to zero,
\[
B_t^\epsilon:=\epsilon^{1/2} \bigl(r_{\epsilon^{-1}t}-
\epsilon^{-1} v_{\star} t \bigr), \qquad t\ge0, %
\]
converges in law on the Skorohod space to a Brownian motion
with variance $\sigma_{\star}^2$.
\end{theorem}

Our strategy to prove Theorems~\ref{t:tcl},~\ref{t:lgn-remanent},
\ref{t:tcl-remanent} is based on the definition of a renewal structure,
extending the approach developed by Comets, Quastel and Ram\'\i rez in~\cite{ComQuaRam} to study the frog model ($D_R>D_B=0$).

Here, a renewal structure is a sequence of a.s. finite random times
$0=:\kappa_0<\kappa_1<\kappa_2<\cdots$ such that:
\begin{itemize}
\item$\mbox{the r.v.s } (\kappa_{n+1}-\kappa_n, r_{\kappa
_{n+1}}-r_{\kappa_n})_{n \geq0}$ are independent,
\item$ \mbox{the r.v.s } (\kappa_{n+1}-\kappa_n, r_{\kappa
_{n+1}}-r_{\kappa_n})_{n \geq1}$ are identically distributed,
\item$\E( \kappa_{2}-\kappa_1)^2 < +\infty$ and $\E( r_{\kappa
_{2}}-r_{\kappa_1} )^2 < +\infty$.
\end{itemize}
Given such a renewal structure, the law of large numbers and the
central limit theorem for $r_t$ can be derived in a standard way,
applying to $r_{\kappa_n}$ the corresponding results for sums of i.i.d.
square-integrable random variables, then approximating $r_t$ by
$r_{\kappa_{n_t}}$, where
$n_t:= \sup\{ n \geq1;   \kappa_n \leq t \}$.

The core of the work lies in finding an appropriate definition for the
renewal structure, and then proving the required tail-estimates. In the
present context, the idea is to find random times $\kappa_n$ that
satisfy the following two conditions: (i) the history of the front
after time $\kappa_n$ does not depend (up to translation) on the future
trajectories of particles located below $r_{\kappa_n}$ at time $\kappa
_n$, and (ii) the distribution of particles located above $r_{\kappa
_n}$ at time $\kappa_n$ is fixed (up to translation).

In \cite{ComQuaRam}, condition (i) is achieved by considering times
after which the front remains forever above a (space--time) straight
line, while particles lying below the front at these times remain
forever below the straight-line. For the frog model, condition (ii) is
then automatically satisfied, since the distribution of blue particles
above the front\footnote{Strictly speaking, this is true only when the
front hits a position above its past record.} is fixed, due to the fact
that blue particles do not move. This is no longer true in the more
complex case when both red and blue particles move, since the
distribution of particles located above $r_t$ at a time where the front
jumps then depends upon the whole past of the process. As a
consequence, new ideas are required to define a proper renewal
structure in this context. We achieve (ii) by extending the
trajectories of our random walks infinitely far into the past, looking
at times before which the front always lies below a straight line,
while particles lying above the front at these times have remained
above the straight line for their whole past history. A key role in the
corresponding argument is played by the invariance properties of the
Poisson distribution of particles, which allows the construction of the
time-reversal of the random walk trajectories and the analysis of the
distribution of the blue particles in terms of this time-reversal.

Once the renewal structure is defined, it is necessary to obtain tail
estimates for the random variables $\kappa_1, r_{\kappa_1}$, and
$\kappa
_{n+1}-\kappa_n$ and $r_{\kappa_{n+1}}-r_{\kappa_n}$ for $n \geq1$. To
this end, we adapt some of the techniques used in \cite{ComQuaRam},
especially the use of martingale methods to control the behavior of
systems of independent random walks. It turns out that some of the more
involved steps in the proof given in \cite{ComQuaRam}, that were needed
to control the accumulation of particles below the front, are replaced
in the present paper by a softer and (hopefully) more transparent argument.

Let us point out one important technical difference between the frog
model and the infection models considered here: ballistic lower bounds
for the position of the front are easy to obtain in the case of the
frog model, while they seem to be very difficult\footnote{By contrast,
ballistic upper bounds are relatively easy to obtain.} for infection
models where both red and blue particles move. In fact, the lower bound
part\footnote{More precisely, a quantitative version of it.} of (\ref
{e:ballistique}) is the main result of \cite{KesSid}, and is obtained
through a quite demanding multi-scale renormalization argument. We do
not provide an independent proof of ballisticity here, and instead have
to rely on the estimate proved in \cite{KesSid}. Still, at least in the
one-dimensional case, our renewal structure approach provides an
alternative way of deriving the law of large numbers (\ref{e:lgn})
(already proved in \cite{KesSid2}) from the coarser ballisticity
estimate obtained in \cite{KesSid}. The original proof in \cite
{KesSid2} is based on an approximate sub-additivity argument, and
relies too on the ballisticity estimates proved in \cite{KesSid}. Note
also that the only missing ingredient to make our proofs of Theorems
\ref{t:lgn-remanent} and~\ref{t:tcl-remanent} work when $D_R>D_B>0$ in
the nonremanent case, is a lower bound on the speed comparable to the
one established in \cite{KesSid} for the single-rate model.\footnote
{Specifically, we would need a result analogous to Proposition~\ref
{p:ballistique-dessous} in Section~\ref{s:estimates}.}

A natural question concerns our specific choice for the Poisson initial
distribution of particles. One can take advantage of the fact that the
random variables $(\kappa_{i+1}-\kappa_i,r_{\kappa_{i+1}}-r_{\kappa
_i})_{i \geq1}$ are independent from the initial configuration of
particles at the left of the origin to show that our results are still
valid if one starts with a Poisson distribution of particles
conditioned upon a nonzero probability event concerning only the
initial configuration of particles at the left of the origin. For
instance, we can prescribe the initial numbers of particles below zero
at any given finite number of sites. Still, it seems necessary to use
the Poisson distribution of particles as a reference probability
measure, so it is unclear how we could extend our results to, say, an
arbitrary initial condition with suitable decay of the number of
particles at infinity.

One should note that, strictly speaking, the initial distribution of
particles we have described is not exactly the same as the one
considered by Kesten and Sidoravicius. Indeed, in \cite{KesSid,KesSid2},
the initial condition is obtained by adding a deterministic finite and
nonzero number of red particles placed arbitrarily, to a configuration
formed by an i.i.d. Poisson number of particles at each site of $\Z$.
For the single-rate KS model on $\Z$, it is irrelevant for the value of
$r_t$ whether particles initially at the left of $r_0$ are red or blue,
so the only difference lies in the added red particles. Using the
previous remark on the possibility to condition the initial
configuration by the numbers of particles at a finite set of sites, we
see that our results in fact include the kind of initial configurations
considered in \cite{KesSid,KesSid2}.

One should also note that the results of \cite{KesSid,KesSid2} are
stated in terms of $\sup_{s \in[0,t]} r_s$ rather than $r_t$ (when
specialized to the one-dimensional case). It clearly makes no
difference for results on the scale of the law of large numbers, since
particles move sub-ballistically. Although such an argument cannot be
used for the central limit theorem, it turns out that, with our
definition of the renewal structure, $r_{\kappa_n} = \sup_{s \in
[0,\kappa_n]} r_s$, so that the CLT holds for either $r_t$ or $\sup_{s
\in[0,t]} r_s$.

Finally, note that our results do not say anything on the case
$D_R<D_B$. The only available results for a model of this kind are
those of \cite{KesSid3}, where a version of the infection model with
$0=D_R<D_B$ is considered, and it is shown that, for sufficiently small
$\rho$, the asymptotic velocity of the front is zero, while it is
conjectured that a positive asymptotic velocity is obtained for
sufficiently large $\rho$.

Let us also mention that other regeneration approaches have been
considered within the context of random walks in dynamic random
environment (see, e.g., \cite{AvedSaVol,dHodSaSid,HildHolSidSoaTei}).

The rest of the paper is organized as follows. In Section~\ref
{s:formal-construction}, we give a formal construction of the random
process associated with the single-rate KS infection model, together
with statements of its main structural properties. Section~\ref
{s:regeneration} provides the definition of the renewal structure, and
its key structural properties are stated and proved there, save for the
estimates on the tail, which form the content of Section~\ref
{s:estimates}. Finally, Section~\ref{s:extension} briefly explains how
to extend the previous results to the case of the remanent KS infection
model with $D_R>D_B$. For the sake of readability, some technical
points are not discussed in detail, and we refer to the arXiv version
\cite{BerRamArxiv} of the present work for a more thorough treatment of
these points.

\section{Formal construction of the single-rate process}\label
{s:formal-construction}

In this section, we describe the construction of the single-rate
process, in two steps. First, we construct, on appropriate spaces, the
dynamics of systems of independent random walks, without any reference
to a possible interaction between them. We then state important
structural properties of the dynamics, such as the strong Markov
property, or the invariance with respect to space--time shifts of the
Poisson distribution on the space of trajectories. Finally, we define
the infection process as a function of these random walks, together
with the corresponding notion of red and blue particles.

\subsection{Reference spaces}

It is convenient to assign a label to each particle in the system, so
that a particle can be uniquely identified by its label. More
precisely, we assume that each particle is labelled by an element of
the interval $[0,1]$, in such a way that no two particles share the
same label. As a consequence, a configuration of particles at a given
time can be represented by a family
\[
w= \bigl(w(x), x \in\Z \bigr),
\]
where, for all $x$, $w(x)$ is a (possibly empty) subset of $[0,1]$,
representing the labels of the particles located at site $x$.

Given $\theta> 0$, introduce the space \index{Stheta@{$\S_{\theta}$}}
$\S_{\theta}$ of all configurations of labelled particles $w=(w(x),  x \in\Z)$ satisfying $w(x) \cap w(y) = \varnothing$ whenever $x \neq y$,
and $\sum_{x \in\Z} \llvert  w(x)\rrvert   e^{-\theta\llvert  x\rrvert  } < +\infty$. Throughout this
paper, $\S_{\theta}$ is our reference space for particle
configurations, where $\theta$ is assumed to be a given positive real
number. The specific value of $\theta$ used in the proofs is made
precise later [see (\ref{e:constraints-param})], and the construction
we now develop is valid for any $\theta> 0$.

To define a distance on $\S_{\theta}$, we first define a distance on
the set of all finite subsets of elements of $[0,1]$. Consider two such
subsets $a= \{ a_1 > \cdots> a_p \}$, and $b= \{ b_1 > \cdots> b_q \}$.
If $p<q$, define $a_i:=0$ for $p+1 \leq i \leq q $;
if $p>q$, define $b_i:=0$ for $ q+1 \leq i \leq p$. Then define the
distance between $a$ and $b$ by
\[
d(a,b):= \llvert q - p \rrvert + \sum_{i=1}^{\max(p,q)}
\llvert b_i-a_i\rrvert.
\]
We now define a distance $d_{\theta}$ on $\S_{\theta}$ by
\[
d_{\theta}(w_1,w_2):= \sum
_{x \in\Z} d \bigl(w_1(x),w_2(x) \bigr)
e^{-\theta\llvert  x\rrvert  }.
\]

Let us turn to the description of particle trajectories. A priori, the
model consists only of particles moving after time zero. However, the
definition of the regeneration structure involves the extension of
their trajectories to negative time indices, so we start from the
beginning with a space allowing the description of trajectories with a
time-index in $\R$. A pair \index{Wu@{$(W,u)$}} $(W,u)$,
where $W=(W_t)_{t \in\R}$ is a c\`adl\`ag function from $\R$ to $\Z$
with nearest-neighbor jumps
(i.e., $\pm1$), and $u \in[0,1]$, is called a (labelled) particle
path, with $u$ being the label of the particle whose path is described
by $W$. In the sequel, we often call such a pair $(W,u)$ a particle,
instead of a particle path.

Given a finite or countable set $\psi$ of particle paths with pairwise
distinct labels, and a time coordinate $t \in\R$, we define the
configuration of labelled particles \index{Xt@{$X_t$}} $X_t(\psi) =
 (X_t(\psi)(x) )_{x \in\Z}$ by
\[
X_t(\psi) (x):= \bigl\{ u; W_t = x, (W,u) \in \psi
\bigr\}.
\]
In words, $X_t(\psi)(x)$ is the set of labels of particle paths that
are located at $x$ at time~$t$. Our reference space for the
trajectories of the particles in the system
is the set $\Omega$ formed by all the sets $\psi$ of particle
trajectories such that $t \mapsto X_t(\psi)$
is a c\`adl\`ag function from $\R$ to $(\S_{\theta}, d_{\theta})$, and
such that no two particle paths jump at the same time.
We endow $\Omega$ with the cylindrical $\sigma$-algebra $\F$ generated
by all the maps $\psi\mapsto X_t(\psi)$ from $\Omega$ to $\S
_{\theta}$
equipped with~the Borel sets associated with the metric $d_{\theta}$.
For all $t \in\R$, we define \index{Ft@{$\F_t$}} $\F_t:= \sigma(X_s,
s \in]-\infty,t])$. For all $x \in\Z$ and $t \in\R$, the space--time
shift \index{pi@{$\pi_{x,t}$}} $\pi_{x,t}$ on $\Omega$ is defined by
the fact that $\pi_{x,t}(\psi)$ is the set of particle paths obtained
from $\psi$ by replacing each path $((W_s)_{s \in\R},u)$ by
$((W_{s-t}-x)_{s \in\R},u)$.
We also consider the space $\DD_+$ as the space of c\`adl\`ag maps from
$[0,+\infty[$ to $\S_{\theta}$. Both spaces are equipped with their
respective cylindrical $\sigma$-algebras.
Finally, we denote by $\Psi$ the canonical map on $\Omega$, that is,
\index{Psi@{$\Psi$}} $\Psi(\psi):= \psi$, so that, whenever we consider
a probability measure $\P$ on $(\Omega, \F)$, the notation $\Psi$
stands for the random set of particle trajectories in the system.

\subsection{Reference probability $\mathbb{P}_w$}

To each $w \in\S_{\theta}$, we associate a probability measure $\P_w$
on $(\Omega,\F)$
describing the evolution of a system of independent particles starting
in configuration $w$ at time $0$.

Fix $w \in\S_{\theta}$, and, for all $x$, write $w(x)$ as an ordered tuple
\[
w(x) = \bigl\{ u(x,1) > \cdots> u \bigl(x, \bigl\llvert w(x) \bigr\rrvert \bigr)
\bigr\},
\]
and define
\[
A:= \bigl\{ (x,i); x \in\Z, 1 \leq i \leq \bigl\llvert w(x) \bigr\rrvert
\bigr\}.
\]
Consider an i.i.d. family of random walks $Z=(Z(x,i),  (x,i) \in A)$
where, for every $(x,i) \in A$, $Z(x,i) = (Z_t(x,i))_{t \in\R}$ is a
two-sided continuous-time random walk on $\Z$, starting at $x$ at time
zero, and evolving in both positive and negative time directions, with
symmetric nearest-neighbor steps, and constant jump rate equal to $2$.
We view $Z(x,i)$ as a random variable taking values in the space of c\`
adl\`ag paths from $\R$ to $\Z$ equipped with the cylindrical $\sigma$-algebra. It can be checked that, up to a modification on a set of
probability zero, the family of random paths $  \{ (Z(x,i),u(x,i));
  (x,i) \in A  \}$ is a random variable taking values in
$(\Omega, \F)$, so that we can define \index{Pw@{$\P_w$}}
\[
\P_w:= \mbox{distribution of } \bigl\{ \bigl(Z(x,i),u(x,i)
\bigr); (x,i) \in A \bigr\}\qquad\mbox{on } (\Omega,\F).
\]
The expectation with respect to $\P_w$ is denoted by $\E_w$.

We now quote two key properties of the family $(\P_w, w \in\S
_{\theta
})$, namely the strong Markov property and the invariance of the
Poisson initial distribution $\P_{\nu}$ with respect to space--time shifts.

%
\begin{prop}\label{p:strong-Markov}
The strong Markov property holds for our process: for every $w \in\S
_{\theta}$, every nonnegative $(\F_t)_{t \geq0}$-stopping time $T$,
and bounded measurable function $F$ on $\DD_+$,
one has that, on $\{ T < +\infty\}$,
%
\begin{equation}
\label{e:strong-Markov}\E_w \bigl( F \bigl((X_{T+t})_{t \geq0}
\bigr) \mid \F_T \bigr) = \E_{X_T} \bigl(F
\bigl((X_t)_{t \geq0} \bigr) \bigr)\qquad \P_w\mbox{-a.s.}
\end{equation}
\end{prop}

We now turn to the invariance properties of the Poisson distribution of
particles with respect to the dynamics. Consider an i.i.d. family
$N=(N_x)_{x \in\Z}$ of Poisson processes on $[0,1]$, with intensity
$\rho$.
With probability one, $(N_x)_{x \in\Z} \in\S_{\theta}$, and we call
\index{nu@{$\nu$}} $\nu$ the probability distribution on $\S
_{\theta}$
induced by $N$. The probability measure $\P_{\nu}$ defined by $\P
_{\nu
}(\cdot):= \int_{\S_{\theta}} \P_w(\cdot) \,d \nu(w)$ is the reference
measure we use to\vspace*{1pt} describe the dynamics starting from a Poisson initial
distribution of particles.

%
\begin{prop}\label{p:shift-invariance}
The probability distribution $\P_{\nu}$ on $\Omega$ is invariant with
respect to the space--time shifts $\pi_{x,t}$.
\end{prop}

\subsection{Infection dynamics}

We now formally define the infection dynamics, through random variables
defined on $(\Omega, \F)$. Given a system of independent random walks
specified by an element of $\Omega$, we define the corresponding front
dynamics, using the fact that particle initially at the left of the
origin are red, while particles initially at the right of the origin
are blue.

We start by defining the sequence $(T_k)_{k \geq0}$, which
characterizes the sequence of times at which the front moves (upward or
downward).
First, let $T_0:=0$, $\mathfrak{r}_{0}:=\sup\{ x \leq0;   \exists
(W,u) \in
\Psi,   W_0=x \}$ (with the convention $\inf\varnothing= -\infty$) and
define inductively the families of random variables
$(T_{\ell})_{\ell\geq0}$ and $(\mathfrak{r}_{\ell})_{\ell\geq0}$
as follows.
Consider $t>T_{\ell}$. We say that $t$ is upward if there exists
$(W,u) \in\Psi$ such that $ W_{t-} = \mathfrak{r}_{\ell}$ and $
W_{t} = \mathfrak{r}_{\ell
}+1$. We say that $t$ is downward if there exists
$(W,u) \in\Psi$ such that $ W_{t-} = \mathfrak{r}_{\ell}$, $W_{t} =
\mathfrak{r}_{\ell
}-1$, and $X_{t-}(\mathfrak{r}_{\ell}) = \{ u \}$.
Then let
\[
T_{\ell+1}:= \inf\{ t > T_{\ell}; \mbox{ $t$ is upward or
downward} \},
\]
with the convention that $\inf\varnothing= +\infty$.
By the fact that paths are c\`adl\`ag in $\S_{\theta}$, one must have
that $T_{\ell+1}>T_{\ell}$ when $T_{\ell}<+\infty$.
Provided that $T_{\ell+1}<+\infty$, one must also have that $T_{\ell
+1}$ is indeed a upward or downward time. In the upward case, we let
$ \mathfrak{r}_{\ell+1}:=\mathfrak{r}_{\ell}+1$. In the downward
case, we let $ \mathfrak{r}_{\ell
+1}:=\mathfrak{r}_{\ell}-1$.
Now \index{r_t@{$r_t$}} $r_t$ is defined on each interval $[T_{\ell},
T_{\ell+1}[$ by $r_t:=\mathfrak{r}_{\ell}$.
From the results in \cite{KesSid}, one has that, for all $k \geq1$,
$T_k < +\infty$, and $\sup_{\ell} T_{\ell} = +\infty$, almost surely
with respect to $\P_{\nu}$.
For the sake of definiteness, we set $\mathfrak{r}_{\ell}:=+\infty$
if $T_{\ell
}=+\infty$, and $r_t:= +\infty$ for $t \geq\sup_{\ell} T_{\ell}$.

In the sequel, we say that a time $t > 0$ is a jump time for the front
if it is one of the times $T_1, T_2,\ldots$ at which the position of
the front
either increases or decreases by one unit.

For all $0<t<+\infty$, we denote by $B_t$ the subfamily of particle
paths corresponding to particles that are blue at time $t$, that is,
\index{Bt@{$B_t$}}
\[
B_t:= \bigl\{ (W,u) \in\Psi; \forall s \in[0, t[,
W_s > r_s \bigr\}.
\]
Similarly, the subfamily of paths associated with particles that are
red at time $t$ is
\index{Rt@{$R_t$}}
\[
R_t:= \bigl\{ (W,u) \in\Psi; \exists s \in[0, t[,
W_s \leq r_s \bigr\}.
\]
We extend the definition by setting $B_0:= \{ (W,u) \in\Psi;  W_0
\geq0 \}$ and $R_0:= \{ (W,u) \in\Psi;  W_0 < 0 \}$. One checks
that, with these definitions, for all $0<t<+\infty$, $r_t$ corresponds
$\P_{\nu}$-a.s. to the position of the
rightmost red particle at time $t$.

In the sequel, we shall use the following $\sigma$-algebras. First,
given $t \geq0$, \index{FRt@{$\F^R_t$}} $\F^R_t$ is defined
by\footnote
{Formally, $\F^R_t$ is generated by all the random variables of the form
\[
\# \bigl( R_t \cap \bigl\{ (W,u); W_s = k, u
\in[a,b] \bigr\} \bigr),
\]
where $k \in\Z$, $0 \leq a < b \leq1$, and $s \leq t$.
}
\[
\F^R_t:= \sigma \bigl( (W_s,u); s \leq t,
(W,u) \in R_t \bigr).
\]
Informally, $\F^R_t$ contains the information relative to the
trajectories of particles that are red at time $t$, up to time $t$.
If $T$ is a nonnegative random variable on $(\Omega, \F)$, we also
define\footnote{Formally, $\F^R_T$ is generated by all the random
variables of the form
\[
\un(s \leq T) \times\# \bigl( R_T \cap \bigl\{ (W,u);
W_s = k, u \in [a,b] \bigr\} \bigr),
\]
where $k \in\Z$, $0 \leq a < b \leq1$, and $s \in\R$.}
\[
\F^R_T:= \sigma(T, r_T) \vee\sigma \bigl(
(W_s,u); s \leq T, (W,u) \in R_T \bigr).
\]
Similarly, we let \index{GRt@{$\G^R_t$}}
\[
\G^R_t:= \sigma \bigl( (W_s,u); s \in\R,
(W,u) \in R_t \bigr).
\]
Informally, $\G^R_t$ contains the information relative to the full
trajectories of the particles that are red at time $t$.
When $T$ is a nonnegative random variable, we also define
\[
\G^R_T:= \sigma(T, r_T) \vee\sigma \bigl(
(W_s,u); s \in\R, (W,u) \in R_T \bigr).
\]

%
%

\section{Regeneration structure}\label{s:regeneration}

\subsection{Definition of \texorpdfstring{$(\kappa_n)_{n\geq0}$}{(kappan){ngeq0}}}
We now define the regeneration structure that is used to prove the
central limit theorem. Remember that it is based on straight lines
drawn on the space--time plane.
In the sequel, $\alpha$ is a strictly positive real number
corresponding to the slope of these straight lines.

%
\begin{figure}[b]

\includegraphics{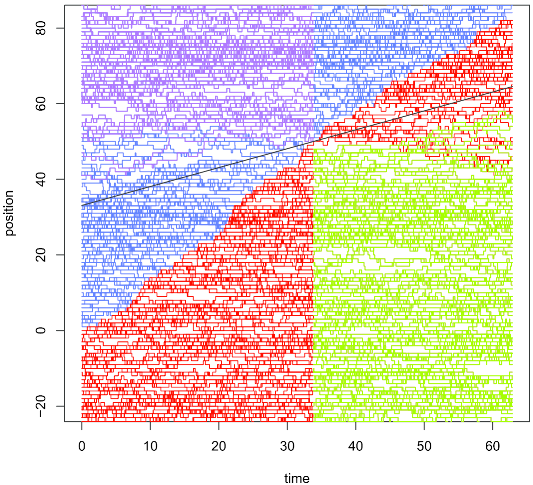}

\caption{A realization of the KS infection model with an $\alpha$
separation time $t$. Posterior (resp., prior) to the forward $\alpha$
time $t$, green (resp., purple) is used instead of red (resp., blue) to
draw the trajectories of particles that lie below (resp., above) $r_t$
at time $t$.\vspace*{-3pt}}\label{f:figure-KS-renewal-5}
\end{figure}

Consider an upward jump time $t > 0$.
We say that $t$ is
a \textit{backward sub-$\alpha$ time} if $r_t > \alpha t$ and if, for all
$0 \leq s < t$, one has $r_s < r_t - \alpha(t-s)$. We say that $t$ is
a \textit{backward super-$\alpha$ time} if, for any $(W,u)$ in $B_t$,
and for all $s < t$, one has $W_s \geq r_t - \alpha(t-s)$. If $t$ is
both a backward sub-$\alpha$ and super-$\alpha$ time,
we say that $t$ is a \textit{backward $\alpha$ time}.
We say that $t$ is a \textit{forward sub-$\alpha$ time} if, for all $(W,u)
\in R_t$ such that $W_t \leq r_t -1$, one has that $W_s \leq
r_t-1+\alpha(s-t)$ for all $s > t$, and if the particle $(W,u)$ making
the front jump at time $t$ remains at $r_t$ during the time-interval
$[t, t+\alpha^{-1}]$, and then satisfies the inequality $W_s \leq
r_t-1+\alpha(s-t)$ for all $s \geq t+ \alpha^{-1}$.
We say that $t$ is a {\it forward super-$\alpha$ time} if, for all $s >
t$, one has $r_s \geq r_t+ \lfloor\alpha(s-t) \rfloor $, and if, moreover,
there exists $(W,u) \in B_t$ such that $W_s = r_t$ for all $s \in[t,
t+\alpha^{-1}]$.
If $t$ is both a forward sub-$\alpha$ and super-$\alpha$ time,
we say that $t$ is a {\it forward $\alpha$ time}.
Finally, if $t$ is both a forward and backward $\alpha$ time, we say
that $t$ is an {\it$\alpha$-separation time}. We extend the definition
of a backward super-$\alpha$ time and of a forward super-$\alpha$ time
by allowing $t=0$ in the above definitions.
These definitions are illustrated in Figure~\ref{f:figure-KS-renewal-5}.\vspace*{-2pt}\vadjust{\goodbreak}

One then defines the renewal structure by letting \index
{kappan@{$\kappa
_n$}} $\kappa_0:=0$, and inductively:
\[
\kappa_{i+1}:= \inf\{T_j > \kappa_i;
\mbox{$T_j$ is an $\alpha $-separation time} \}.
\]

Before discussing why the above definition indeed yields a renewal
structure for the model, let us briefly explain why it is at least
conceivable that such a sequence of $\alpha$-separation times exists.
First, note that the existence of backward sub-$\alpha$ times is a
direct consequence of the front moving ballistically, provided that
$\alpha$ is chosen in such a way that $\alpha< \liminf t^{-1} r_t$.
Similarly, ballisticity of the front with speed strictly greater than
$\alpha$ also yields the existence of forward super-$\alpha$ times. On
the other hand, for a system of independent random walks whose
distribution at time $0$ is characterized by i.i.d. Poisson numbers of
particles at every site $x \leq0$, the maximum position occupied at
time $t \geq0$ by a random walk in the system, grows only sub-linearly
as a function of $t$. This provides at least a heuristic justification
of why forward sub-$\alpha$ times exist, and a symmetric argument can
be made for backward super-$\alpha$ times, by invoking time-reversal
and the reversibility of the Poisson distribution of particles with
respect to the dynamics of independent random walks. With a mild dose
of faith, the simultaneous occurrence of these four properties at a
single time $t$ should look plausible. Mathematical arguments giving
rigorous content to this heuristic line of reasoning are found in
Section~\ref{s:estimates}.

\subsection{Key properties of \texorpdfstring{$(\kappa_n)_{n\geq1}$}{(kappan){ngeq1}}}

The key properties of the sequence $(\kappa_{n})_{n \geq1}$ are stated
in the following theorem.

%
\begin{theorem}\label{t:structure-renouvellement}
With respect to $\P_{\nu}$, the r.v.s $(\kappa_n)_{n \geq0}$ are a.s.
finite and:
\begin{itemize}
\item$\mbox{the r.v.s } (\kappa_{n+1}-\kappa_n, r_{\kappa
_{n+1}}-r_{\kappa_n})_{n \geq0}$ are independent,
\item$ \mbox{the r.v.s } (\kappa_{n+1}-\kappa_n, r_{\kappa
_{n+1}}-r_{\kappa_n})_{n \geq1}$ are identically distributed,
\item$\E( \kappa_{2}-\kappa_1)^2 < +\infty$ and $\E( r_{\kappa
_{2}}-r_{\kappa_1} )^2 < +\infty$.
\end{itemize}
\end{theorem}
Given Theorem~\ref{t:structure-renouvellement}, it is more or less
standard to derive Theorem~\ref{t:tcl}, approximating $r_t$ by
$r_{\kappa_{n_t}}$, where
$n_t:= \sup\{ n \geq0;   \kappa_n \leq t \}$. Note that, due to the
definition of $\kappa$,
one has $r_{\kappa_{n_t}} \leq r_t \leq r_{\kappa_{n_t+1}}$, which
eases the corresponding approximation argument. We do not give the
details here (see, e.g., \cite{ComQuaRam}).

The proof of Theorem~\ref{t:structure-renouvellement} relies on two
distinct results, stated below as Propositions~\ref{p:renouvel-1} and
\ref{p:moments-renouvel}. Proposition~\ref{p:renouvel-1} deals with
structural properties of $(\kappa_n)_{n \geq0}$, while Proposition
\ref
{p:moments-renouvel} provides tail estimates.

%
\begin{prop}\label{p:renouvel-1}
For all $n \geq1$, one has the following properties:
\begin{longlist}[(a)]
\item[(a)] the r.v.s $\kappa_1,\ldots, \kappa_n$ and $r_{\kappa
_1},\ldots, r_{\kappa_n}$ are measurable with respect to $\G
^R_{\kappa_n}$.
\item[(b)] on $\{ \kappa_n < +\infty\}$, the conditional distribution
of $(\kappa_{n+1} - \kappa_n, r_{\kappa_{n+1}}- r_{\kappa_{n}})$ with
respect to $\G^R_{\kappa_n}$ is the distribution\footnote{This is a
slight abuse of terminology, since, strictly speaking, $B_0$ is only
$\P
_{\nu}$-a.s. equal to a random variable from $(\Omega, \F)$ to itself.}
of $(\kappa_1, r_{\kappa_1})$ $(B_0)$
with respect to $\P_{\nu}$, conditioned on $t=0$ being a backward and
forward super-$\alpha$ time for $B_0$.
\end{longlist}
\end{prop}
Note that, in the formulation of Proposition~\ref{p:renouvel-1} above,
the fact that $t=0$ is a backward and forward super-$\alpha$ time for
$B_0$ means that the assumptions characterizing a backward and forward
super-$\alpha$ time are satisfied
when the set of particle trajectories taken into account is restricted
to $B_0$, that is, with $\Psi$ replaced by $B_0$.
In particular, the relevant front dynamics here is $(r_t(B_0))_{t \geq
0}$, not $(r_t)_{t \geq0}$. Also, implicit in this formulation is the
fact that the conditioning event has a nonzero probability, which is
proved in Corollary~\ref{c:prob-super-alpha}.

%
\begin{prop}\label{p:moments-renouvel}
For small enough $\alpha$ (depending on $\rho$), there exists $\theta
>0$ such that $\E_{\nu}(\kappa_1^2)<+\infty$ and $\E_{\nu
}(r_{\kappa
_1}^2)<+\infty$.
\end{prop}

Deducing Theorem~\ref{t:structure-renouvellement} from Propositions
\ref
{p:renouvel-1} and~\ref{p:moments-renouvel} is straightforward. The
rest of this section is devoted to the proof of Proposition~\ref
{p:renouvel-1}, while Proposition~\ref{p:moments-renouvel}, is proved
in Section~\ref{s:estimates}.

\subsection{Structural properties: Proof of Proposition \texorpdfstring{\protect\ref{p:renouvel-1}}{3}}

For the sake of readability, the proof of Proposition~\ref
{p:renouvel-1} is divided into a sequence of four steps.

Step 1 establishes the measurability condition (a) in Proposition~\ref
{p:renouvel-1}. This is a classical step when dealing with renewal
structures, although a nontrivial one since the $\kappa_n$s look
infinitely far into the future of the trajectories. It merely reflects
the consistency across the $\kappa_n$s of the various comparison
conditions involving parallel space--time lines of slope $\alpha$. Step
2 is similar, establishing that, broadly speaking, going from $\kappa
_n$ to $\kappa_{n+1}$ is equivalent to going from $0$ to $\kappa
_1$, keeping only the trajectories of particles that are blue at time
$\kappa_n$ and applying a space--time translation sending $(r_{\kappa
_n},\kappa_n)$ to $(0,0)$. Step 3 explicitly characterizes the
distribution of particle trajectories that are blue at time $T_k$,
conditional on the trajectories of particles that are red at time
$T_k$. This is a key result, relying on the invariance properties of
the distribution $\P_{\nu}$. Step 4 builds on this result to
characterize the distribution of particle trajectories that are blue at
time $\kappa_n$, conditional on the trajectories of particles that are
red at time $\kappa_n$.

\subsubsection{Step 1: Measurability with respect to $\G^R_{\kappa
_n}$} We now prove statement (a) in Proposition~\ref{p:renouvel-1},
that is, the fact that, for all $n \geq1$, the r.v.s $\kappa_1,\ldots,
\kappa_n$ and $r_{\kappa_1},\ldots, r_{\kappa_n}$ are measurable with
respect to $\G^R_{\kappa_n}$.

First, note that the measurability of $\kappa_n$ and $r_{\kappa_n}$
with respect to $\G^R_{\kappa_n}$
is a direct consequence of the definition of $\G^R_{\kappa_n}$. Also,
with our conventions, the result is obvious on $\{ \kappa_n = + \infty
\}$, so we may assume that
$\{ \kappa_n < + \infty\}$.

From the definition of the infection dynamics,
particle paths $(W,u)$ outside $R_{\kappa_n}$ have no influence on the
front jumps between time $0$ and $\kappa_n$,
so that the history of the front up to time $\kappa_n$ is exactly the
same as the one that would be obtained
if there were no other particle paths in the system besides those in
$R_{\kappa_n}$. As a consequence,
the jump times $T_1 < \cdots< T_{\ell} = \kappa_n$
that lie between time $0$ and $\kappa_n$, are measurable with respect
to $\G^R_{\kappa_n}$. What remains to be proved is that, for every jump
time $T_i$ such that $1 \leq i \leq\ell-1$,
it is possible to tell whether $T_i$ is a backward/forward
sub/super-$\alpha$ time,
using only the information contained in $\G^R_{\kappa_n}$, which is not
{a priori} obvious since the definition of each $\kappa_1,\ldots,
\kappa_{n-1}$, imposes some conditions on every particle trajectory in
the system, so we have to check that, as far as particles in $B_{\kappa
_n}$ are concerned, these conditions are subsumed by those already
imposed by the definition of $\kappa_n$:
\begin{itemize}
\item Whether $T_i$ is a backward sub-$\alpha$ time involves only the
history of the front up to time $\kappa_n$, so this condition is $\G
^R_{\kappa_n}$-measurable.
\item Whether $T_i$ is a backward super-$\alpha$ time involves
conditions on trajectories in $B_{T_i} \cap R_{\kappa_n}$, which are
$\G
^R_{\kappa_n}$-measurable, and conditions on trajectories in $B_{T_i}
\cap B_{\kappa_n}$
which are automatically satisfied thanks to the fact that $\kappa_n$
itself is a backward super-$\alpha$ time and the fact that $r_{T_i}
\leq r_{\kappa_n} - \alpha(\kappa_n-T_i)$ since $\kappa_n$ is also a
backward sub-$\alpha$ time.
\item Whether $T_i$ is a forward sub-$\alpha$ time involves the
trajectories of paths in $R_{T_i} \subset R_{\kappa_n}$ only, so this
condition is $\G^R_{\kappa_n}$-measurable.
\item Whether $T_i$ is a forward super-$\alpha$ time involves a
condition of the front up to time $\kappa_n$, so this condition is $\G
^R_{\kappa_n}$-measurable, which is $\G^R_{\kappa_n}$-measurable, plus
a condition on the front posterior to $\kappa_n$ which is automatically
satisfied thanks to the fact that $\kappa_n$ itself is a forward
super-$\alpha$ time.
\end{itemize}

\subsubsection{Step 2: From $\kappa_n$ to $\kappa_{n+1}$}

We now state the following property: for all $n \geq1$ $\{ \kappa_n
<+\infty\}$, the following identity holds:
%
\begin{equation}
\label{e:representation-kappa} (\kappa_{n+1}-\kappa_n, r_{\kappa_{n+1}}
- r_{\kappa_n}) = (\kappa_1, r_{\kappa_1}) \bigl( \pi
_{r_{\kappa_n},\kappa_n}( B_{\kappa_n}) \bigr).
\end{equation}
The meaning of the above identity is that $(\kappa_{n+1}-\kappa_n,
r_{\kappa_{n+1}} - r_{\kappa_n})$ is identical to $(\kappa_1,
r_{\kappa
_1})$ applied to $\pi_{r_{\kappa_n},\kappa_n}( B_{\kappa_n})$, that is,
the system consisting only of trajectories of particles that are blue
at time $\kappa_n$, to which a space--time translation has been applied
so that $(r_{\kappa_n},\kappa_n)$ is sent to $(0,0)$. The fact that the
only trajectories playing a role are those in $B_{\kappa_n}$ is a
consequence of the forward $\alpha$ time property of $\kappa_n$:
particles that are red at time $\kappa_n$ do not have any influence on
the evolution of the front after time $\kappa_n$, since their
trajectories are confined below a space--time line of slope $\alpha$,
while the front is constrained to lie above this same line. Checking
the backward/forward sub/super-$\alpha$ time conditions in a way
similar to the proof of Step~1 above, precisely leads to identity (\ref
{e:representation-kappa}) (the details can be found in \cite{BerRamArxiv}).

\subsubsection{Step 3: Distribution of $\pi_{r_{T_k},T_k}( B_{T_k})$
given $\G^R_{T_k}$ under $\mathbb{P}_{\nu}$} We now establish the fact that,
under $\P_{\nu}$, the conditional
distribution of $\pi_{r_{T_k},T_k}(B_{T_k})$ with respect to $\G
^R_{T_k}$, is the same as that of $B_0$, conditioned by the fact that
every trajectory in $B_0$ avoids the space--time translated trajectory
of the front, that is, $ (r_{T_k+t}-r_{T_k} )_{-T_k \leq t < 0}$.
More formally, we have the following.
%
\begin{prop}\label{p:conditionnement-1}
Let $F: \Omega\to\R$ denote a bounded measurable map. Then, for
all $k \geq1$, on the event that $T_k$ is upward,
\[
\E_{\nu} \bigl(F \bigl( \pi_{r_{T_k},T_k} ( B_{T_k})
\bigr) \mid \G^R_{T_k} \bigr) = \xi \bigl(F,
(r_{s+T_k}-r_{T_k})_{-T_k \leq s \leq0}, -r_{T_k} \bigr)\qquad\mbox{a.s.},
\]
%

%
\begin{figure}[b]

\includegraphics{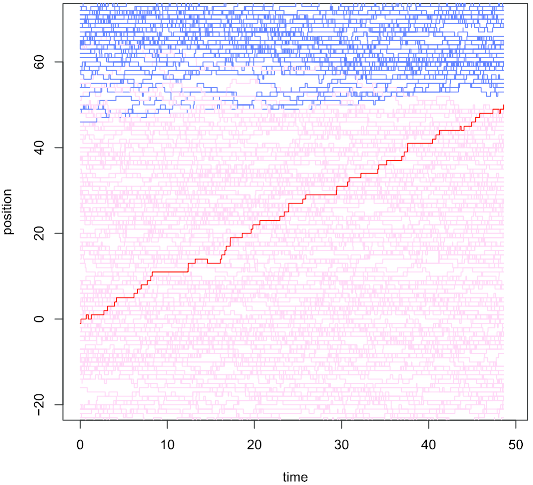}

\caption{Realization of the KS infection model. Pink trajectories
correspond to particles that are red \textit{at time $t$}, that is, $R_t$,
while blue trajectories correspond to particles that are blue {\it at
time $t$}, that is, $B_t$. The trajectory of the front is drawn in red.}
\label{f:simu-KS-retourne-2}
\end{figure}

\noindent
where,\footnote{Note that the definition makes sense since, as is
easily checked, $\P_{\nu}(G(q,x))>0$ for all c\`adl\`ag path
$q=(q_s)_{t \leq s \leq0}$ with values in $\Z$, taking
nearest-neighbor steps, such that $q_0 = 0$, $q_{0-}=-1$,
and containing a finite number of jumps.} given a path $q=(q_s)_{t \leq
s \leq0}$ with values in $\Z$,
\[
\xi(F,q,x):= \E_{\nu} \bigl( F(B_0) \mid G(q,x) \bigr),
\]
and
\index{Gqx@{$G(q,x)$}}
\begin{equation}
\label{e:def-Gqx}G(q,x):= \bigl\{ \forall(W,u) \in B_0,
W_t > x, \forall t \leq s < 0, W_s > q_s
\bigr\}.
\end{equation}
\end{prop}

Figure~\ref{f:simu-KS-retourne-2} illustrates this conditioning.
Note that the avoidance condition obviously has to be satisfied.
Indeed, trajectories of particles that are blue at time $T_k$ must have
avoided the front during the time-interval $[0,T_k[$, for a particle in
contact with the front at a time $t<T_k$ will necessary be red at time
$T_k$. What is not obvious is that the influence upon $B_{T_k}$ of the
history of the whole process up to time $T_k$, admits such a simple description.

The core idea underlying the proof of Proposition~\ref
{p:conditionnement-1} is time-reversal. Instead of starting from the
description of the initial configuration at time zero, and trying to
understand how it evolves up to time $t$, we start from the
configuration at time~$t$, and express the relevant properties in terms
of the time-reversed trajectories. This is where the reversibility of
the Poisson distribution of particles with respect to systems of
independent random walks plays a key role. Broadly speaking, one starts
from a configuration at time $t$ with i.i.d. Poisson numbers of
particles at each site, and a prescribed value, say $x$, for $r_t$. One
can then express the whole history between time $0$ and time $t$ of
particles that are red at time $t$---including the trajectory of the
front from time $0$ to time $t$---in terms of random walk trajectories
going backward from those sites that are $<x$. On the other hand, the
history of particles that are blue at time $t$ is described by an {\it
independent} set of random walk trajectories going backward from sites
$\geq x$, with the constraint that they must avoid the front between
time $0$ and time $t$. This informal description is made precise in the
proof of Proposition~\ref{p:conditionnement-1} given below.

\begin{pf*}{Proof of Proposition~\ref{p:conditionnement-1}}
To avoid the main idea being obscured by technicalities, we start with
a simplified argument, which is not completely valid since it involves
conditioning upon the exact values taken by continuous random
variables. If our model were a discrete-time one, this argument would
immediately translate into a full proof. In our continuous-time
framework, additional approximation arguments are needed, of which we
give a sketch, referring to \cite{BerRamArxiv} for a full account of
the technical details.

Consider $y \in\Z$, $t>0$, $v \in[0,1]$, and a c\`adl\`ag path
$g=(g_s)_{0 \leq s \leq t}$ with values in $\Z$, taking
nearest-neighbor steps, such that $g_t = y$, $g_{t-}=y-1$. Now define
the event $J=J(t,y,v,g)$ by
\[
J = \bigl\{ T_k = t, r_{T_k} = y, U_k = v,
{(r_s)_{0 \leq s
\leq t} = (g_s)_{0 \leq s \leq t}} \bigr
\},
\]
where $U_k$ denotes the (random) label of the particle making the front
move at time~$T_k$. Then introduce a partition of the set of particle
paths defined by
\[
\Delta_+ = \bigl\{ (W,u) \in\Psi; W_{t} > y \mbox{ or }
(W_t=y \mbox{ and }u \neq v) \bigr\}
\]
and $\Delta_- = \Psi\setminus\Delta_+$. Finally, let $(r'_s)_{0
\leq
s \leq t}$ denote the position of the front generated by the particles
in $\Delta_-$ up to time $t$, $T'_k$ the $k$th time at which this front
moves, and $U'_k$ the label of the particle path making this front move
at time $T'_k$. Introduce the events $J'$ and $A$ defined by
\begin{eqnarray*}
J' &=& \bigl\{ T'_k = t,
r'_{T'_k} = y, U'_k = v, {
\bigl(r'_s \bigr)_{0 \leq s
\leq t} = (g_s)_{0 \leq s \leq t}}
\bigr\},
\\
A &=& \bigl\{ \forall(W,u) \in\Delta_+, W_0 > 0 \mbox{ and } \forall0
\leq s < t, W_s > g_s \bigr\}.
\end{eqnarray*}
The key identity to prove Proposition~\ref{p:conditionnement-1} is the
following:
%
\begin{equation}
\label{e:llave} J= J' \cap A.
\end{equation}

We first check the inclusion $J \subset J' \cap A$. Note that, on $J$,
by definition $\Delta_+$ coincides with $B_t$. As a consequence, the
two avoidance conditions in $A$ have to be satisfied since particles in
$\Delta_+$ have to be blue at time $t$. On the other hand, on~$J$, we
also have that $\Delta_-$ coincides with $R_t$. Since the history of
the front up to time $t$ is entirely prescribed by the dynamics of
particles that are red at time $t$, the quantities $r'_t, T'_k, U'_k$
must then coincide with $r_t, T_k, U_k$.

We now check the reverse inclusion $J' \cap A \subset J$. Let us prove
that, on $J$, one has $r'_s=r_s$ for all $0 \leq s \leq t$. By
contradiction, assume that there exists a time $s \leq t$ such that
$r'_s \neq r_s$, and let $s_0$ be the first such time. Due to the
condition $W_0>0$ for particles in $\Delta_+$, we must have $r'_0=r_0$,
so that $s_0>0$. Then the only possibility for $r'_{s_0}$ not to be
equal to $r_{s_0}$ is that some particle in $\Delta_+$ either makes
$r_s$ (and not $r'_s$) jump upward at time $s_0$, or prevents $r_s$
(and not $r'_s$) to jump downward at time $s_0$. In turn, this implies
that such a particle hits the position $r'_s$ at some time $s \leq
s_0$, which is ruled out by the definition of $A$. Knowing that
$r'_s=r_s$ for all $0 \leq s \leq t$, the other conditions in $J$ are
automatically satisfied.

Now, on $J=J' \cap A$, we have seen that $B_t=\Delta_+$ and
$R_t=\Delta
_-$. On the other hand, in view of the definition of $\Delta_+$ and
$\Delta_-$, the invariance of the probability $\P_{\nu}$ with respect
to space--time shifts (Proposition~\ref{p:shift-invariance}) entails that
%
\begin{equation}
\label{e=transl-inv-Delta} \bigl(\pi_{y,t}(\Delta_+),\pi _{y,t}(
\Delta_-) \bigr) \stackrel{d} {=} (B_0,R_0),
\end{equation}
and we note that $B_0$ and $R_0$ are independent due to the Poisson
structure of $\P_{\nu}$. Finally, we note that, from their definitions,
$J'$ is measurable w.r.t. $\Delta_-$ while $A$ is measurable w.r.t.
$\Delta_+$,
with the explicit representation in terms of $ \pi_{y,t}(\Delta_+)$:
\[
A= \bigl\{ \forall(W,u) \in\pi_{y,t}(\Delta_+), W_{-t} > -y, \forall-t \leq s' < 0, W_{s'} > g_{-t-s'}-y
\bigr\}.
\]

If $J$ were an event with nonzero probability, we would readily deduce
from the above results that, conditional upon $J=J' \cap A$, the random
variable $\pi_{r_{T_k},T_k}(B_{T_k} ) = \pi_{y,t}(\Delta_+)$ is
independent from $R_{T_k}=\Delta_-$, and follows the distribution of
$B_0$ conditioned upon the event
\[
\bigl\{ \forall(W,u) \in B_0, W_{-t} > -y, \forall-t
\leq s' < 0, W_{s'} > g_{-t-s'}-y \bigr\}.
\]
Moreover, from a countable decomposition into pairwise disjoint events
of the form
\[
\{ T_k \mbox{ is upward} \} = \bigsqcup
_{t,y,v,g} J(t,y,v,g),
\]
we would then deduce the conclusion of the proposition. However, in our
continuous-time framework, each event of the form $J(t,y,v,g)$ has zero
probability, and a countable decomposition as above does not exist. To
tackle this problem, we rely on discrete approximations, which allow us
to recover the conclusion in the limit. More precisely, letting
$T^{(\ell)}_k:= 2^{-\ell}(\lceil2^{\ell} T_k \rceil )$ and
$U^{(m)}_k:=
2^{-m}(\lceil2^{m} U_k \rceil )$, we can perform a countable decomposition
based on the values of $T^{(\ell)}_k$ and $U^{(m)}_k$, but it then
becomes necessary to show that the contributions of various undesirable
events (e.g., between time $T_k$ and $T^{(\ell)}_k$) lead to a
vanishing contribution when we take the limits $\ell,m \to+\infty$.
Moreover, to get decent convergence properties, we have to use
regularity properties of the Markov semigroup of $(X_t)_t$, and
characterize the conditional distribution of $\pi_{r_{T_k},T_k}(B_{T_k}
)$ through the conditional expectation of sufficiently nice functionals
$F$, the appropriate choice being $F=f_1(X_{t_1}) \times\cdots\times
f_p(X_{t_p})$,
where $t_1<\cdots< t_p$, and $f_1,\ldots, f_p$ are bounded and
uniformly continuous on $(\S_{\theta},d_{\theta})$.
\end{pf*}

\subsubsection{Step 4: Distribution of $\pi_{r_{\kappa_n},\kappa_n}(B_{\kappa_n})$ given $\G^R_{\kappa_n}$ under $\mathbb{P}_{\nu}$}
We now establish the fact that, under $\P_{\nu}$, on the event $\{
\kappa
_n < +\infty\}$, the conditional
distribution of $\pi_{r_{\kappa_n},\kappa_n}(B_{\kappa_n})$ with
respect to $\G^R_{\kappa_n}$, is the same as that of $B_0$, conditioned
by the event $H$ defined as\footnote{Corollary~\ref{c:prob-super-alpha}
proves that $\P_{\nu}(H)>0$.}
\begin{equation}
\label{e:def-H}H:= \{ \mbox{$t=0$ is a backward and forward super-$\alpha$ time
for $B_0$} \}.
\end{equation}
More formally, we prove that, for all bounded measurable map $F:
\Omega\to\R$, one has that, on $\{\kappa_n < +\infty\}$,
%
\begin{equation}
\label{e:caract-kappa}\E_{\nu} \bigl(F \bigl( \pi_{r_{\kappa
_n},\kappa_n} (
B_{\kappa_n}) \bigr) \mid \G^R_{\kappa_n} \bigr) =
\E_{\nu} \bigl(F(B_0) \un_H \bigr) \qquad\mbox{a.s.}
\end{equation}

The key idea to go from Step~3 to the present result consists in
showing that, when $T_k=\kappa_n$, the avoidance condition that results
from conditioning $B_{T_k}$ by $R_{T_k}$, is subsumed by the condition
that $\kappa_n$ is both a backward super- and sub-$\alpha$ time.

Consider $C \in\G^R_{\kappa_n}$ such that $C \subset\{ \kappa_n <
+\infty\}$, and write the decomposition
%
\begin{equation}
\label{e:decoupe-esp}\E_{\nu} \bigl(F \bigl( \pi_{r_{\kappa
_n},\kappa_n} (
B_{\kappa_n}) \bigr) \un_C \bigr) = \sum
_{k \geq1} \E_{\nu} \bigl(F \bigl( \pi_{r_{T_k},T_k} (
B_{T_k}) \bigr) \un_{C} \un(\kappa_n=T_k)
\bigr).
\end{equation}

We first note that, for each $k \geq1$, the event $\{ \kappa_n = T_k
\}
$ corresponds to $t=0$ being a backward and forward super-$\alpha$ time
for $\pi_{r_{T_k},T_k}( B_{T_k})$, plus a\vspace*{1pt} set of conditions bearing
only on $R(T_k)$, that is, the trajectories of particles that are red
at time $T_k$, and implying among other things that $T_k$ is a backward
sub-$\alpha$ time.
As a consequence,
one can write $\{ \kappa_n = T_k \} = H_k \cap J_k$,
where\vspace*{1pt} $J_k \in\G^R_{T_k}$ and
$H_k:= \{ \mbox{$t=0$ is a backward and forward super-$\alpha$ time 
for $\pi_{r_{T_k},T_k} ( B_{T_k})$} \}$,
and with $J_k \subset\{ T_k$ is a backward sub-$\alpha$ time$\}$.
Moreover, since $C \in\G^R_{\kappa_n}$, one can write $C \cap\{
\kappa_n = T_k \} = D_k \cap\{ \kappa_n = T_k \}$,
where $D_k \in\G^R_{T_k}$. Putting things together, we obtain the identity
%
\begin{equation}
\label{e:encore-long}
\quad\E_{\nu} \bigl(F \bigl( \pi_{r_{T_k},T_k} (
B_{T_k}) \bigr) \un_{C} \un(\kappa_n=T_k)
\bigr) = \E_{\nu} \bigl(F \un_H \bigl( \pi
_{r_{T_k},T_k} ( B_{T_k}) \bigr) \un_{D_k \cap J_k} \bigr).
\end{equation}
%
We can now invoke Proposition~\ref{p:conditionnement-1}, using the fact
that $D_k,J_k \in\G^R_{T_k}$, leading to the identity
%
\begin{equation}
\label{e:encore-long-2} \E_{\nu} \bigl(F \un_H \bigl( \pi
_{r_{T_k},T_k} ( B_{T_k}) \bigr) \un_{D_k \cap J_k} \bigr) =
\E_{\nu}(K_k \un _{D_k \cap J_k}),
\end{equation}
where we have set $K_k:=\xi(F \un_H, (r_{s+T_k}-r_{T_k})_{-T_k \leq s
\leq0}, -r_{T_k})$.

We now make the key observation that, for a path
$q=(q_s)_{t \leq s \leq0}$ such that $q_0 = 0$ and $q_s < \alpha s$
for all $t \leq s < 0$, and $x$ such that $x < \alpha t$, the event
that $t=0$ is a backward super-$\alpha$ time implies the event
$G(q,x)$, so that in particular $H \subset G(q,x)$. As a consequence,
we can write
%
\begin{equation}
\label{e:encore-long-3}\xi(F \un_H,q,x)=\E_{\nu}
\bigl(F(B_0) \un_H(B_0) \mid G(q,x) \bigr)
= \frac{\E_{\nu}(F(B_0) \un_H)}{\P_{\nu}(G(q,x))}.
\end{equation}
By the fact that, on $J_k$, $T_k$ is a backward sub-$\alpha$ time, we
can precisely apply the above observation to the path $q=
(r_{s+T_k}-r_{T_k})_{-T_k \leq s \leq0} $ and the position
$x=-r_{T_k}$. Putting together (\ref{e:decoupe-esp}), (\ref
{e:encore-long}), (\ref{e:encore-long-2}), (\ref{e:encore-long-3}), we
obtain that
%
\begin{equation}
\label{e:fini}\E_{\nu} \bigl(F \bigl( \pi_{r_{\kappa_n},\kappa
_n} (
B_{\kappa_n}) \bigr) \un_C \bigr) = \E_{\nu}
\bigl(F(B_0) \un_H \bigr) \times h(C),
\end{equation}
where $h(C)$ is an expression depending on the event $C$ but not on
$F$. Using (\ref{e:fini}) with the choice $F \equiv1$ shows that
$h(C)=\P_{\nu}(C)$, so that (\ref{e:fini}) indeed proves identity
(\ref{e:caract-kappa}).

\subsubsection{Conclusion}
Part (a) of Proposition~\ref{p:renouvel-1} has been proved in Step~1.
As for part (b), we know from Step~2 that, on the event $\{ \kappa_n <
+\infty\}$, we have
\[
(\kappa_{n+1}-\kappa_n, r_{\kappa_{n+1}} -
r_{\kappa_n}) = (\kappa_1, r_{\kappa_1}) \bigl(
\pi_{r_{\kappa_n},\kappa_n}( B_{\kappa_n}) \bigr).
\]
From Step~4, the conditional distribution of $\pi_{r_{\kappa
_n},\kappa
_n}( B_{\kappa_n})$ with respect to $\G^R_{\kappa_n}$ is that of $B_0$
conditioned by the event that $t=0$ is a
backward and forward super-$\alpha$ time for $B_0$, which yields the
desired result.\vadjust{\eject}

\section{Estimates on the renewal structure}\label{s:estimates}

\subsection{Overview}

This section is devoted to the proof of Proposition~\ref
{p:moments-renouvel}. To control the tail of the random variables
$\kappa_1$ and $r_{\kappa_1}$, we rely on a sequence of stopping times
$S_1 \leq D_1 \leq D_2 \leq S_2 \leq\cdots,$ where $S_n$ is the $n$th
attempt at obtaining an \mbox{$\alpha$-}separation time, and, in case this
attempt fails, $D_n$ is the time at which the failure is detected,
while $D_n=+\infty$ if the attempt is successful. As a consequence, one
has that $\kappa_1 \leq S_{\K}$, where $\K:= \inf\{ n \geq1;   D_n
= +\infty\}$, and our approach consists in bounding the tail of the
number $\K$, and the tail of the increments $S_{n+1}-S_n$ and
$r_{S_{n+1}}-r_{S_n}$ on the event that $D_n<+\infty$.

The organization of this section is the following. In Section~\ref
{ss:def-SD}, we give the precise definition of the random variables
$S_n$ and $D_n$, while Section~\ref{ss:param} lists the various
parameters, assumptions and conventions, used in subsequent estimates.
In Section~\ref{ss:hitting-straight}, we prove elementary results on
the hitting times and hitting probabilities of a straight line by a
system of independent random walks. Section~\ref{ss:ball-est} is
devoted to an extension of the quantitative ballisticity estimates
obtained in \cite{KesSid} to the case where the initial distribution of
particles is restricted to locations above $x=0$. Section~\ref
{ss:cond-dist-Sn} contains an analogue of Proposition~\ref
{p:conditionnement-1} suited to the definition of the stopping time
$S_n$. Finally, Section~\ref{ss:tails} combines these ingredients to
prove the tail estimates on the renewal structure, that is, Proposition
\ref{p:moments-renouvel}.

\subsection{Definition of $S_n$ and $D_n$}\label{ss:def-SD}

We define by induction the sequence of stopping times on which our
estimates on the renewal structure are based. Besides $\alpha$, the
definition involves two integer parameters
$\mathscr{C} \geq1$ and $L \geq1$, and the following notion: given $0
\leq s < t$, we say that $t$ is an $(s,\alpha)$-\textit{crossing time} if
there exists $k \in\{1, 2, \ldots\}$ such that $r_v < r_s + k +
\alpha
(v-s)$ for all $v \in[s,t[$
and $r_t \geq r_s + k + \alpha(t-s)$.

To initialize the induction, let $D_0:=0$ and $\Upsilon_0:= \varnothing
$. Now, for $n \geq1$, assume that the random variables $D_{n-1},
\Upsilon_{n-1}$ have already been defined, and let \index{Spn@{$S'_n$}}
$S'_{n}$ be the infimum of the $t > D_{n-1}$ such that:
\begin{itemize}
\item$t$ is a backward sub-$\alpha$ time;
\item$\Upsilon_{n-1} \subset R_t$;
\item$B_t$ contains at least $\mathscr{C}$ particles $(W,u)$ such that
$W_t=r_t$.
\end{itemize}
Then define \index{Sn@{$S_n$}}
$S_n$ as the infimum of the $t > S'_{n}$ such that:
\begin{itemize}
\item$t$ is a backward sub-$\alpha$ time;
\item$]S'_{n}, t [$ contains a number of $(S'_n, \alpha)$-crossing
times at least equal to $L$;
\item$B_t$ contains at least $\mathscr{C}$ particles $(W,u)$ such that
$W_t=r_t$.
\end{itemize}

We use the notation \index{Wnst@{$(W^{*n}, u^{*n})$}} $(W^{*n},
u^{*n})$ for the particle that makes the front jump at time~$S_n$,
and define the subset \index{RSnst@{$R_{S_n}^*$}} $R_{S_n}^*:= R_{S_n}
\setminus\{ (W^{*n}, u^{*n}) \}$.
If $S_{n}$ is a backward super-$\alpha$ time, then $\Upsilon_n:=
\varnothing$ and \index{Dn@{$D_n$}} $D_n$ is defined as the infimum of
the $t>S_n$ such that {\it a least one} of the following five
conditions holds:
\begin{longlist}[(2)]
\item[(1)] $r_t < r_{S_n} + \lfloor\alpha(t-S_n) \rfloor $,
\item[(2)] $ t \leq S_n+\alpha^{-1}$ and there is no $(W,u) \in
B_{S_n}$ such that $W_{S_n} = r_{S_n}$ and $W$ remains at $r_{S_n}$
during $[S_n, S_n+t]$,
\item[(3)] $W_t > r_{S_n}-1 + \alpha(t-S_n)$ for some $(W,u) \in R_{S_n}^*$,
\item[(4)] $t \leq S_n+\alpha^{-1}$ and $W^{*n}_t \neq r_{S_n}$,
\item[(5)] $ t > S_n+\alpha^{-1}$ and $W^{*n}_t > r_{S_n} -1 + \alpha
(t-S_n)$.
\end{longlist}
Note that (1) and (2) detect the potential failure of $S_n$ to be a
forward super-$\alpha$ time, while (3)--(4)--(5) detect the potential
failure of $S_n$ to be a forward sub-$\alpha$ time.

On the other hand, if $S_{n}$ is not a backward super-$\alpha$ time,
consider the set of particle paths $(W,u) \in B_{S_n}$ such that there
exists $t < S_n$ for which $W_t < r_{S_n} - \alpha(S_n-t)$.
Among this set, consider the pair \index{Wn@{$(W^{(n)}, u^{(n)})$}}
$(W^{(n)},u^{(n)})$ such that $(W_{S_n},u)$ is the smallest with
respect to the lexicographical order,\footnote{Remember that
$(x_1,u_1)$ is smaller than $(x_2,u_2)$ with respect to the
lexicographical order if $x_1 < x_2$, or $x_1=x_2$ and $u_1<u_2$.} and define
$\Upsilon_n:= \{ (W^{(n)},u^{(n)}) \}$
and $D_{n}:= S_n$, so that, when nonempty, the set $\Upsilon_n$ can
be though of as containing a witness trajectory for the fact that $S_n$
failed to be a backward super-$\alpha$ time.

The reason why $S_n$ is not taken equal to $S'_n$ is a technical one,
and comes from the fact that, when proving tail estimates, one needs to
have some ``room'' between $S'_n$ and $S_n$ so that, at time $S_n$, the
configuration of particles below the front has ``smoothed out'' the
irregularities that may be present at time $S'_n$.

Some of the above definitions are illustrated in Figures~\ref
{f:figure-struct-KS-1},~\ref{f:figure-struct-KS-3} and~\ref
{f:figure-struct-KS-4}.

%
\begin{figure}[t]

\includegraphics{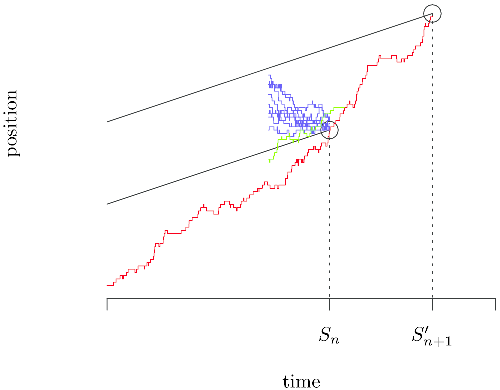}

\caption{From $S_n$ to $S'_{n+1}$ when $S_n$ fails to be a backward
super-$\alpha$ time (in this case $D_n=S_n$). Only the most relevant
portions of trajectories are shown. The trajectory of the front $r_t$
is depicted in red, while blue is used for the trajectories of blue
particles, except for the witness trajectory $W^{(n)}$, which is drawn
in green. Circles are used at locations where the number of particles
is assumed to be $\geq\mathscr{C}$.}
\label{f:figure-struct-KS-1}
\end{figure}

%
\begin{figure}[b]

\includegraphics{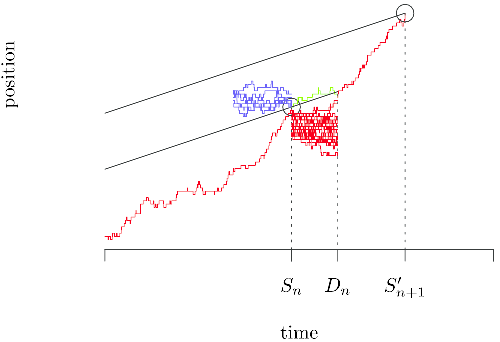}

\caption{From $S_n$ to $S'_{n+1}$ when $S_n$ is a backward
super-$\alpha
$ time but condition (1) is realized first. Only the most relevant
portions of trajectories are shown. The trajectory of the front $r_t$
is depicted in red, except for the part causing condition (1), which is
drawn in green. Circles are used at locations where the number of
particles is assumed to be $\geq\mathscr{C}$.}
\label{f:figure-struct-KS-3}
\end{figure}

%
\begin{figure}

\includegraphics{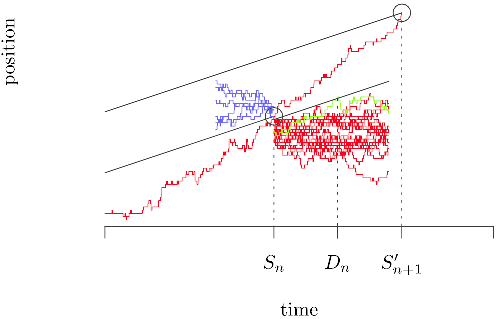}

\caption{From $S_n$ to $S'_{n+1}$ when $S_n$ is a backward
super-$\alpha
$ time but condition (3) is realized first. The trajectory of the front
$r_t$ is depicted in red, which is also used for red particles, except
the trajectory causing condition (3), which is drawn in green. Circles
are used at locations where the number of particles is assumed to be
$\geq\mathscr{C}$.}

\label{f:figure-struct-KS-4}
\end{figure}

\subsection{List of parameters, assumptions and conventions}\label{ss:param}

Let us recapitulate the list of parameters encountered so far:
$D_R=D_B>0$ is the common jump rate of red and blue particles, $\rho>0$
is the average number of particles per site in the initial Poisson
distribution of particles, $\theta>0$ is a parameter characterizing the
space $\S_{\theta}$ of particle configurations we work with, $\alpha>
0$ is the slope of space--time lines involved in the definition of the
renewal structure, $\mathscr{C}$ and $L$ are two additional integer
parameters involved in the definition of $S_n$ and $D_n$ given above.
While $D_R=D_B$ and $\rho$ are fixed parameters of the model, $\theta,
\alpha, \mathscr{C}, L$ can be chosen at our convenience. One more
parameter $\beta>0$ will play a
role in the following proofs.

Here are the assumptions on the various parameters that we assume to
hold throughout the sequel:
%
\begin{equation}
\label{e:constraints-param} \cases{ 0<\alpha< \beta< (1/3) C_2(\rho/4),
\vspace*{3pt}\cr
\alpha
\theta - 2(\cosh\theta-1) > 0,}
\end{equation}
where $C_2>0$ is defined in Proposition~\ref{p:ballistique-dessous-KS},
which is adapted from \cite{KesSid}. Such a choice of parameters is
always possible by choosing first $\alpha$ and $\beta$, then $\theta$
close enough to zero, using the fact that $\cosh(\theta)=1+o(\theta)$
when $\theta$ goes to zero. In addition to (\ref{e:constraints-param}),
we shall have to assume that $\mathscr{C}$ is large enough, and also
that $L$ is large enough (depending on $\mathscr{C}$). These
assumptions on $\mathscr{C}$ and/or $L$ will always be made explicit in
the sequel.

We now explain our convention for constants: what we call constants in
the rest of this section may depend on $\rho, \alpha, \beta, \theta$,
but unless otherwise mentioned, not on $\mathscr{C}$ or $L$. As a rule,
we use $c_1,c_2,\ldots$ to denote constants whose range of validity
extends throughout the section, and are used in the statement of
propositions or lemmas. On the other hand, we use $d_1,d_2,\ldots$ to
denote constants that are purely local to proofs.

\subsection{Hitting of a straight line by random walks}\label{ss:hitting-straight}

Starting with an initial configuration of particles $w \in\S_{\theta
}$, we establish two bounds on the hitting time and probability of a
straight line of slope $\alpha$ by one of the random walks whose
initial position is $\leq0$. In both cases, the key quantity is the
``exponential norm'' $\phi_{\theta}$ defined by
\[
\phi_{\theta}(w):= \sum_{x \leq0} \sum
_{u \in w(x)} e^{\theta x}.
\]
Lemma~\ref{l:controle-proba-atteinte} below gives (a) an upper bound on
the probability that the hitting time is finite and $\geq t$, and shows
(b) that an upper bound on the value of $\phi_{\theta}(w)$ translates
into a lower bound on the probability that none of the random walks
ever hits the straight line.

%
\begin{lemma}\label{l:controle-proba-atteinte}
The following bounds hold:
\begin{longlist}[(a)]
\item[(a)] For all $w \in\S_{\theta}$, and all $t \geq0$,
\[
\P_{w} \bigl( \exists(W,u)\ \exists s \geq t, W_0 \leq0,
W_s \geq\alpha s \bigr) \leq\phi_{\theta}(w) e^{-\mu t},
\]
where $\mu:= \alpha\theta- 2(\cosh\theta-1)>0$ [see (\ref
{e:constraints-param})].
\item[(b)] For all $K > 0$, there exists $g(K)>0$ such that, for all $w
\in\S_{\theta}$
such that $\phi_{\theta}(w) \leq K$, one has
\[
\P_{w} \bigl( \forall(W,u) \mbox{ such that }W_0 \leq0,
\forall t > 0, \mbox{ one has } W_t < \alpha t \bigr) \geq g(K).
\]
\end{longlist}
\end{lemma}

The proof of the lemma is based on the following elementary result for
a single random walk, which we state and prove first.

%
\begin{lemma}\label{l:atteinte-droite}
Let \index{zetas@{$\zeta_s$}} $(\zeta_s)_{s \geq0}$ be
a continuous-time simple symmetric random walk on $\Z$
with total jump rate $2$ starting at $x \leq0$, with respect to a
probability measure~$P_x$.
Then for all $t \geq0$
\[
P_x ( \exists s \geq t; \zeta_s \geq\alpha s ) \leq
e^{\theta x} e^{- \mu t}.
\]
\end{lemma}
\begin{pf}
For all $s \geq0$, let $M_s:=e^{\theta\zeta_s - 2 (\cosh(\theta)
-1)s}$, and $T:= \inf\{ s \geq t;   \zeta_s \geq\alpha s \}$.
Then $(M_s)_{s \geq0}$ is a c\`adl\`ag martingale, and $T$ is a
stopping time, so that, for all finite $K>0$, one has
%
\begin{equation}
\label{e:Doob-1} E_x(M_{T \wedge K}) = E_x(M_0)
= e^{\theta x}.
\end{equation}
Now we have that $\liminf_{K \to+\infty} M_{T \wedge K} \geq M_{T}
\un
(T<+\infty)$, so that, by Fatou's lemma and (\ref{e:Doob-1}),
%
\begin{equation}
\label{e:magic-martingale-1}E_x \bigl( M_{T} \un(T<+\infty ) \bigr)
\leq e^{\theta x}.
\end{equation}
Now, by definition of $T$, one has that, on $\{ T < +\infty\}$,
%
\begin{equation}
\label{e:magic-martingale-2}M_T \geq e^{ \theta\alpha
T - 2 (\cosh(\theta) -1)T} = e^{\mu T} \geq
e^{\mu t},
\end{equation}
where the last inequality comes from the fact that $\mu> 0$ and $T
\geq t$.
The result now follows from combining (\ref{e:magic-martingale-1}) and
(\ref{e:magic-martingale-2}).
\end{pf}

\begin{pf*}{Proof of Lemma~\ref{l:controle-proba-atteinte}}
First note that part (a) of the lemma is a direct consequence of Lemma
\ref{l:atteinte-droite} and of the union bound over each particle.
We now prove part~(b), using the notation already appearing in the
statement of Lemma~\ref{l:atteinte-droite}.

Let us choose $\theta'>\theta$ such that $\mu':= \alpha\theta' - 2
(\cosh(\theta') -1) > 0$ (this is possible since $\mu>0$), and observe
that Lemma~\ref{l:atteinte-droite} holds with $\theta', \mu'$ instead
of $\theta, \mu$. We deduce that, for all $x < 0$, we have
%
\begin{equation}
\label{e:proba-atteinte-prime}\mathfrak{p}(x) \leq e^{\theta' x},
\end{equation}
where
\[
\mathfrak{p}(x):= P_x ( \exists s > 0; \zeta_s \geq
\alpha s ).
\]
Moreover, we must have $\mathfrak{p}(0) < 1$, for otherwise we could
prove that
\[
P_0 \Bigl( \limsup_{t \to+\infty} \zeta_t/t
\geq\alpha \Bigr) = 1,
\]
which would contradict the law of large numbers.
Since all the random walks in our model evolve independently, we can
rewrite the probability we want to bound from below as
\[
\prod_{x \leq0} \bigl( 1 - \mathfrak{p}(x)
\bigr)^{\llvert  w(x)\rrvert  }.
\]
Now the inequality $\phi_{\theta}(w) \leq K$ implies that, for all $x
\leq0$, one has that
%
\begin{equation}
\label{e:borne-somme-exp} \bigl\llvert w(x) \bigr\rrvert \leq e^{- \theta x} K.
\end{equation}
As a consequence, we have the bound
\[
\prod_{x \leq0} \bigl( 1 - \mathfrak{p}(x)
\bigr)^{\llvert  w(x)\rrvert  } \geq \biggl( \prod_{x
\leq0}
\bigl(1 - \mathfrak{p}(x) \bigr)^{ e^{- \theta x} K} \biggr).
\]
In view of (\ref{e:proba-atteinte-prime}) and of the fact that $\theta
'> \theta$, we have that $\sum_{x \leq0} e^{- \theta x} e^{\theta' x}
< +\infty$, so the right-hand side of the above inequality is $>0$, and
depends only on $K$.
\end{pf*}

\subsection{Ballisticity estimates}\label{ss:ball-est}

We start by recalling the quantitative ballisticity estimates derived
in \cite{KesSid}.
%
\begin{prop}\label{p:ballistique-dessus-KS}
There exist a constant \index{C1rho@{$C_1(\rho)$}} $C_1(\rho)>0$ and a
constant $\Clconstcccpenki$, depending on $\rho$ and $\mathscr{C}$, such
that, for every $t > 0$,
\[
\P_{\nu} \bigl(r_t \geq C_1(\rho) t \bigr)
\leq\Clconstcccpenki \exp(-t).
\]
\end{prop}
%
%
\begin{prop}\label{p:ballistique-dessous-KS}
There exists a constant \index{C2rho@{$C_2(\rho)$}} $C_2(\rho)>0$ such
that, for all $K > 0$, there exists a constant $\Clconstcccnulis$,
depending on $\rho$ and $K$, such that, for every $t > 0$,
\[
\P_{\nu} \bigl(r_t \leq C_2(\rho) t \bigr)
\leq \Clconstcccnulis t^{-K}.
\]
\end{prop}
Note that, strictly speaking, Propositions~\ref
{p:ballistique-dessus-KS} and~\ref{p:ballistique-dessous-KS} do not
appear in \cite{KesSid}, which uses slightly different definitions for
the front and the initial condition. However, they are rather easily
derived from Theorems 1 and 2 in \cite{KesSid}.

When trying to control the tail of $\kappa_1$, we encounter situations
where ballisticity estimates similar to those appearing in Propositions
\ref{p:ballistique-dessus-KS} and~\ref{p:ballistique-dessous-KS} are
needed, but where the initial condition consists only of particles
located at the right of the origin. To be specific, we define $\nu
_{\mathscr{C},+}$ \index{nuCp@{$\nu_{\mathscr{C},+}$}} to be the
probability distribution on $\S_{\theta}$ obtained by starting from the
Poisson distribution $\nu$, removing every particle whose location is
$<0$, and conditioning the number of particles located at $x=0$ to be
$\geq\mathscr{C}$. The corresponding distribution on the space of
trajectories is denoted by $\P_{\nu_{\mathscr{C},+}}$.

Adapting the upper bound (Proposition~\ref{p:ballistique-dessus-KS})
with $\nu_{\mathscr{C},+}$ instead of $\nu$ turns out to be rather
straightforward, since removing red particles from the initial
condition cannot increase the position of the front [see equation (\ref{e:monotonie})].
The precise result we need in the sequel is an easy corollary from this
adaptation, and we quote it without proof.
%
\begin{prop}\label{p:ballistique-dessus-KS-extension}
Let \index{Cp1rho@{$C'_1(\rho)$}} $C'_1(\rho):=C_1(\rho)+1$. There
exist strictly positive constants $\Clconstcccseptyni , \Clconstcccdesimt $,
with $\Clconstcccseptyni$ depending on $\mathscr{C}$, such that, for every $t > 0$,
\[
\P_{\nu_{\mathscr{C},+}} \bigl(\exists s \geq0; r_s \geq
C'_1(\rho) \max (s,t) \bigr) \leq\Clconstcccseptyni \exp(-
\Clconstcccdesimt t).
\]
\end{prop}
On the other hand, adapting the lower bound (Proposition~\ref
{p:ballistique-dessous-KS}) requires more work. The precise result we
have is the following, with $\beta>0$ being defined in~(\ref
{e:constraints-param}).
%
\begin{prop}\label{p:ballistique-dessous}
There exist constants $\Clconstcccvienas , \Clconstcccdu  >0$, with $\Clconstcccvienas$ depending on $\mathscr{C}$, such that, for every $t > 0$,
\[
\P_{\nu_{\mathscr{C},+}}( \exists s \geq t; r_s \leq\alpha s ) \leq
\Clconstcccvienas t^{-\Clconstcccdu \cdot\mathscr{C}}.
\]
\end{prop}
The following corollary to Proposition~\ref{p:ballistique-dessous},
whose proof is given in the last part of the present subsection, shows
that the super-$\alpha$ time condition for $B_0$, which constitutes
``half'' of the conditions involved in the definition of the renewal
structure, has indeed a positive probability with respect to $\P_{\nu
_{\mathscr{C},+}}$.
%
\begin{coroll}\label{c:prob-super-alpha}
For all large enough $\mathscr{C}$,
\[
\P_{\nu_{\mathscr{C},+}} \bigl(\{\mbox{$t=0$ is a backward and forward super-$\alpha$
time for $B_0$} \} \bigr) > 0.
\]
\end{coroll}

The rest of this subsection is mainly devoted to the proof of
Proposition~\ref{p:ballistique-dessous}, which is based on coupling the
evolution of the front with respect to $\P_{\nu_{\mathscr{C},+}}$ with
a modified version of the dynamics, then using a symmetrization trick
to enable a comparison with an initial configuration consisting of
i.i.d. Poisson numbers of particles on the whole of $\Z$.

\subsubsection{Step 1: Monotone couplings}

We start by defining a modified version of the infection dynamics. In
the modified version, the front is at zero at time zero and, after time
zero, the dynamics is defined as the original one, with the difference
that the front is never allowed to go below level zero (i.e., a jump
that would make the front go below zero for the original dynamics has
no effect on the front in the modified dynamics). We call $(\hat
{r}_s)_{s \geq0}$ the trajectory of the corresponding front.

Our first statement is that both the original front $r_t$ and the
modified front $\hat{r}_t$ are nondecreasing with respect to the
addition of particles in the system. Indeed, we claim that, for all
$\psi_1,\psi_2 \in\Omega$ such that $\psi_1 \subset\psi_2$, one has,
for all $t \geq0$,
%
\begin{equation}
\label{e:monotonie} r_t(\psi_1) \leq r_t(
\psi_2)
\end{equation}
and
%
\begin{equation}
\label{e:monotonie-modif} \hat{r}_t(\psi_1) \leq \hat
{r}_t(\psi_2).
\end{equation}
The proof of (\ref{e:monotonie}) consists in observing that, by
definition, $r_0(\psi_1) \leq r_0(\psi_2)$, and that, since only
nearest-neighbor steps can be performed, the trajectories of both
fronts must meet before crossing each other. Assuming that at a time
$s$ one has $r_s(\psi_1) = r_s(\psi_2)$, and calling $t$ the next time
at which any of the fronts jumps, the assumption that $\psi_1 \subset
\psi_2$ entails that, if $t$ is upward for $\psi_1$, it is also upward
for $\psi_2$, while, if $t$ is downward
for $\psi_2$, it is also downward for $\psi_1$, so that $r_t(\psi_1)
\leq r_t(\psi_2)$ in any case. We argue similarly to prove (\ref
{e:monotonie-modif}).

Next, we claim that the modified front always dominates the original
front, in the sense that, for all $\psi\in\Omega$, one has, for all
$t \geq0$,
%
\begin{equation}
\label{e:croissance} r_t(\psi) \leq\hat{r}_t(\psi),
\end{equation}
which can be proved with an argument quite similar to that used for
(\ref{e:monotonie}).

Finally, we prove the monotonicity of the modified front with respect
to a symmetrization of trajectories that we implement through a map
$\mathscr{T}:  \Omega\to\Omega$.
Consider a pair $(W,u)$. If
$W_0 \geq0$, then let $W^{\times}:=W$. On the other hand, if $W_0<0$, consider
$\tau:= \inf\{ s > 0;   W_s = 0 \}$, and let $W^{\times}_s:=-W_s$ on
$]{-}\infty, \tau[$ and
$W^{\times}_s:=W_s$ on $[\tau,+\infty[$. Now let
\[
\mathscr{T}(\psi):= \bigl\{ \bigl(W^{\times}, u \bigr); (W,u) \in\psi
\bigr\}.
\]
We can now state the monotonicity property: for all $\psi\in\Omega$,
and for all $t \geq0$,
%
\begin{equation}
\label{e:croissance-modif} \hat{r}_t(\psi) \leq\hat {r}_t \circ
\mathscr{T}(\psi).
\end{equation}
To prove (\ref{e:croissance-modif}), we argue as in the proof of (\ref
{e:monotonie}), so that it is enough to prove that, if $\hat{r}_s =
\hat
{r}_s \circ\mathscr{T}$, and if $t$ denotes the first time after $s$
at which any of the fronts jumps, one has $\hat{r}_t \leq\hat{r}_t
\circ\mathscr{T}$.
Assume that $t$ is upward for $\hat{r}$. Then by definition the
corresponding random walk $W$ is such that $W_s \geq0$, so that
$W^{\times}$ coincides with $W$ on $[s, +\infty[$, and so $t$ is
also upward for $\hat{r} \circ\mathscr{T}$. On the other hand, if $t$
is downward for $\hat{r} \circ\mathscr{T}$, then the common location
of the fronts has to be $\geq1$, and there must be at least
one $(W,u) \in\Psi$ such that $W_s = W^{\times}_s = \hat{r}_s$. In
fact, there cannot be more than one such $(W,u)$, since otherwise $t$
could not be downward for $\hat{r} \circ\mathscr{T}$. As a
consequence, there is only one such $(W,u)$, and $t$ must also be
downward for $\hat{r}$.

\subsubsection{Step 2: Comparison between distributions}

Given $t \geq0$, let $t_0:=t/3$, $t_1:= (2/3)t$, and, for $s \in
[t_1,t]$, define $r^{(1)}_s:= \hat{r}_{s-t_1} \circ\pi_{0, t_1}$.
Then define four distributions $\nu_+,\nu_1,\nu_2,\nu_3$ on $\S
_{\theta
}$, in the following way. First, let $\nu_+$ be the probability
distribution on $\S_{\theta}$ obtained by starting from the Poisson
distribution $\nu$, then removing every particle whose location is
$<0$. Then let $(N^{(1)}_x)_{x \in\Z}$ denote an independent family of
Poisson processes on $[0,1]$,
where, for all $x \in\Z$, the rate of $N^{(1)}_x$ is equal to $\rho
p_{t_1}(x, \N)$, with
$p_{t_1}(x,\N):= \sum_{y \in\N} p_{t_1}(x,y)$. Define $\nu_1$ as the
distribution induced by
$(N^{(1)}_x)_{x \in\Z}$ on $\S_{\theta}$. Define also $(N^{(2)}_x)_{x
\in\Z}$ to be an\vspace*{1pt} independent family of Poisson processes on $[0,1]$,
where the rate of $N^{(2)}_x$ is $\rho/2$ for $x \geq1$, $\rho/4$ for
$x = 0$, and $0$ for $x < 0$, and define $\nu_2$ as the distribution
induced by $(N^{(2)}_x)_{x \in\Z}$. Finally, define $\nu_3$ exactly as
$\nu$, with the difference
that the constant value of the rate is equal to $\rho/4$ instead of
$\rho$.

We now claim that
%
\begin{equation}
\label{e:compare-final} {r}_{t_0}( \P_{\nu_3} ) \prec
r^{(1)}_{t}( \P_{\nu_{\mathscr{C}, +}}),
\end{equation}
where $\prec$ denotes stochastic domination between probability
measures on $\R$. We also use stochastic domination on $\S_{\theta}$
equipped with the order relation induced by inclusion between sets,
that is, $w_1 \leqslant w_2$ when $w_1(x) \subset w_2(x)$ for all $x
\in\Z$.

To begin with, one checks that $\nu_{+}$ is stochastically dominated by
$\nu_{\mathscr{C},+}$. As a consequence, the distribution of $X_{t_1}$
with respect to $\P_{\nu_{\mathscr{C}, +}}$ stochastically dominates~$\nu_1$. Using (\ref{e:monotonie-modif}), we deduce that
%
\begin{equation}
\label{e:comp-loi-1} r^{(1)}_{t}( \P_{\nu_{+}}) \prec
r^{(1)}_{t}( \P_{\nu_{\mathscr{C}, +}}).
\end{equation}
Then observe that $\nu_1$ is the distribution of $X_{t_1}$ with
respect to
$\P_{\nu_{+}}$, so that
%
\begin{equation}
\label{e:comp-loi-0} \hat{r}_{t_0}( \P_{\nu_1} ) =
r^{(1)}_{t}( \P_{\nu_{+}}).
\end{equation}
Then $\nu_2$ is stochastically dominated by $\nu_1$, since, for all $x
\geq0$, we have $p_{t_1}(x, \N) \geq1/2$. By (\ref
{e:monotonie-modif}), we deduce that
%
\begin{equation}
\label{e:comp-loi-2}\hat{r}_{t_0}( \P_{\nu_2} ) \prec
\hat{r}_{t_0}( \P_{\nu_1} ).
\end{equation}
We also have that the image of the probability measure $\P_{\nu_3}$ by
the map $\mathscr{T}$ is $\P_{\nu_2}$, so that, by (\ref
{e:croissance-modif}),
%
\begin{equation}
\label{e:comp-loi-3} \hat{r}_{t_0}( \P_{\nu_3} ) \prec
\hat{r}_{t_0}( \P_{\nu_2} ).
\end{equation}
Using (\ref{e:croissance}), we finally deduce that
%
\begin{equation}
\label{e:comp-loi-4} {r}_{t_0}( \P_{\nu_3} ) \prec \hat
{r}_{t_0}( \P_{\nu_3} ).
\end{equation}
Putting together (\ref{e:comp-loi-1}), (\ref{e:comp-loi-0}), (\ref
{e:comp-loi-2}), (\ref{e:comp-loi-3}), (\ref{e:comp-loi-4}), we see that
(\ref{e:compare-final}) is proved.

\subsubsection{Step 3: Sojourn above zero}

To successfully exploit the comparisons established in Step~2, we need
to control the probability, with respect to $\P_{\nu_{\mathscr{C},
+}}$, that the front remains for a substantial amount of time above the
origin. Our claim is that there exist constants $\Clconstccavienas$,
$\Clconstccadu >0$, with $\Clconstccavienas$ depending on $\mathscr{C}$, such
that, for every $t > 0$,
%
\begin{equation}
\label{e:au-dessus-de-zero} \P_{\nu_{\mathscr{C}, +}} \Bigl(\inf_{s \in[(2/3) t,t]}
r_s \leq0 \Bigr) \leq\Clconstccavienas t^{-\Clconstccadu \cdot\mathscr{C}}.
\end{equation}

We now proceed to the (essentially self-contained and elementary) proof
of~(\ref{e:au-dessus-de-zero}). Let $t_0:=t/3$. Then fix a real number
$0<v<\sqrt{2/3}$, and define
$y_{t_0}:= \lfloor v (t_0 \log t_0)^{1/2} \rfloor $ and $\varepsilon(t_0):=
\frac{t_0^{-v^2/4}}{v (\log t_0)^{1/2}}$. Let\vspace*{1pt} $(\zeta_s)_{s \geq0}$
denote a continuous-time simple
symmetric random walk with total jump rate $2$ starting at site $x$,
with respect to a probability measure $P_x$.
By a standard local limit theorem,\footnote{See, for example, \cite
{Fel}, Chapter~XVI, Section~6 on Large Deviations for the case of a discrete-time random walk. The
continuous-time follows easily, by controlling the fluctuations of the
number of steps performed by the walk.}
we have that, as $t$ goes to infinity,
%
\begin{equation}
\label{e:tll}P_0(\zeta_{t_0} \leq-y_{t_0})
\sim \Clconstddcvienas  \varepsilon(t_0),
\end{equation}
where $\Clconstddcvienas$ is a positive constant. Using the reflection principle,
we deduce that there exists a strictly positive constant $\Clconstddbv$ such that, for large $t$,
\[
P_0 \Bigl( \inf_{s \in[0,t_0] } \zeta_s
\leq-y_{t_0} \Bigr) \leq \Clconstddbv \varepsilon(t_0).
\]
Now let $\mathfrak{Z}_s$ denote the supremum of the positions at time
$s$ of the particle paths that are
located at the origin at time zero, and let $C_1$ denote the event that
$\mathfrak{Z}_s > -y_{t_0}$ for all $s \in[0,t_0]$.
Since the number of these particle paths is at least
$\mathscr{C}$, we deduce that
%
\begin{equation}
\label{e:borne-C1c} \P_{\nu_{\mathscr{C}, +}} \bigl(C_1^c \bigr)
\leq \Clconstddbv^{\mathscr{C}} \varepsilon(t_0)^{\mathscr{C}}.
\end{equation}
Now let $z_{t_0}:= \lfloor \varepsilon(t_0)^{-3} \rfloor $, and
consider the
number $\NNN$ of particle paths whose location at time zero lies in the
interval $[0, z_{t_0}]$. Let $C_2$ denote the event that $\NNN$ is at
least equal to
$\rho z_{t_0}/2$. By standard large deviations bounds for Poisson
random variables (see, e.g., \cite{DemZei}), we have that, for all
large $t$,
%
\begin{equation}
\label{e:borne-C2c}\P_{\nu_{\mathscr{C}, +}} \bigl(C_2^c \bigr) \leq
\exp(-\Clconstddcptr z_{t_0}),
\end{equation}
for some strictly positive constant $\Clconstddcptr$.
Now define $\NNN'$ to be the number of particle paths that:
\begin{longlist}[(a)]
\item[(a)] start at an initial position in $[0, z_{t_0} ]$;
\item[(b)] hit $-y_{t_0}$ during the time-interval $[0,t_0]$;
\item[(c)] hit $0$ after having hit $-y_{t_0}$ and before time $2 t_0$.
\end{longlist}
For a particle starting in $[ 0, z_{t_0} ]$, the probability to hit
$-y_{t_0}$ during $[0,t_0]$
is larger than or equal to $q_{t_0}:= P_{z_{t_0}}  ( \inf_{s \in
[0,t_0] } \zeta_s \leq-y_{t_0}  )$. Moreover, using
the symmetry of the walk, we see that, starting from $-y_{t_0}$, the
probability for the walk to hit $0$ before time
$t_0$ is larger than or equal to $q_{t_0}$. As a consequence, given
$\NNN$, the distribution of $\NNN'$ stochastically
dominates a binomial distribution with parameters $\NNN$ and
$q_{t_0}^2$. Moreover, as $t$ goes to infinity, $z_{t_0}=o(t_0^{1/2})$
and $y_{t_0} z_{t_0} = o(t_0)$ due to the fact that
$v^2<2/3$, so that (\ref{e:tll}) is also valid for $P_{z_{t_0}}(\zeta
_{t_0} \leq-y_{t_0})$, from which we deduce that, for large $t$,
\[
q_{t_0} \geq\Clconstddbke  \varepsilon(t_0),
\]
where $\Clconstddbke $ is a strictly positive constant.
Define $C_3$ to be the event that $\NNN' \geq\NNN q_{t_0}^2/2$.
Using standard (see, e.g., \cite{McD}) large deviations bounds for
binomial random
variables, we deduce from the preceding discussion that for all large
enough $t$,
%
\begin{equation}
\label{e:borne-C3c}\P_{\nu_{\mathscr{C}, +}} \bigl(C_2 \cap C_3^c
\bigr) \leq\exp \bigl(- \Clconstddcppe \varepsilon(t_0)^{-1}
\bigr),
\end{equation}
for some strictly positive constant $\Clconstddcppe$.
Now consider the intervals of the form $[2t_0+k, 2t_0+k+1]$, for $0
\leq k \leq\lfloor t_0 \rfloor $. Then
consider a random walk satisfying conditions (a) to (c) above, stopped
at the first time it hits the origin
after having hitted $-y_{t_0}$; by definition, this time is $\leq
2t_0$. By symmetry, the probability that this walk
is above $0$ at time $2t_0+k$ is $\geq1/2$, and the probability that
it then remains above $0$ during the whole interval
$[2t_0+k, 2t_0+k+1]$ is larger than some
strictly positive constant $\Clconstddcppps$. As a consequence, for
each of the intervals we consider, the probability
that none of the random walks that satisfy (a) to (c) lies above zero
for the duration of the interval is conditional
upon $\NNN'$, bounded above by $(1-\Clconstddcppps)^{\NNN'}$. Now define
$C_4$ as the event that, for every $s \in[2t, 3t]$, there
exists at least one random walk satisfying (a) to (c) whose position at
time $s$ is $\geq0$.
Using the union bound over all the intervals, whose total number is
$\leq t_0+1$, we obtain that for all large enough $t$,
%
\begin{equation}
\label{e:borne-C4c}\P_{\nu_{\mathscr{C}, +}} \bigl(C_2 \cap C_3
\cap C_4^c \bigr) \leq(t_0+1) \exp \bigl(-
\deptyni  \varepsilon (t_0)^{-1} \bigr).
\end{equation}

We now observe that, on $C_1 \cap C_2 \cap C_3 \cap C_4$, one must have
$r_s \geq0$ for all $s \in[2t_0, 3t_0] = [(2/3)t, t]$.
Indeed, we know that the front always lies above the maximum position
of the particles initially at zero.
By $C_1$, the front lies above $-y_{t_0}$ during the interval
$[0,t_0]$. As a consequence, any particle path satisfying
(a)~and~(b) must hit the front before time $t_0$. For that reason, on
$C_4$, the front lies above $0$ during the interval
$[2t_0, 3t_0]$.
Now using (\ref{e:borne-C1c}), (\ref{e:borne-C2c}), (\ref
{e:borne-C3c}), (\ref{e:borne-C4c}), we
have that, for large enough $t$, the probability of the complement of
$C_1 \cap C_2 \cap C_3 \cap C_4$ is bounded above by
\[
\Clconstddbv^{\mathscr{C}} \varepsilon(t_0)^{\mathscr{C}} + \exp(-
\Clconstddcptr z_{t_0})+ \exp \bigl(- \Clconstddcppe \varepsilon(t_0)^{-1}
\bigr) + (t_0+1) \exp \bigl(-\deptyni  \varepsilon(t_0)^{-1}
\bigr),
\]
and the first term dominates the others when $t_0$ is large.

\subsubsection{Conclusion}

We now put together the different pieces leading to the proof of
Proposition~\ref{p:ballistique-dessous}. Remember from Step~2 the
notation $t_1 = (2/3)t$ and $r^{(1)}_s = \hat{r}_{s-t_1} \circ\pi
_{0, t_1}$,
and let $C:= \{ r_s \geq0$ for all $s \in[t_1, t ] \}$.
Our first claim is that
%
\begin{equation}
\label{e:comp-reste-positif} \mbox{on the event $C$, one has that $r^{(1)}_s
\leq r_s$\qquad for all $s \in[t_1, t ]$.}
\end{equation}
Indeed, on $C$, one has that $r^{(1)}_{t_1} \leq r_{t_1}$ since by
definition $r^{(1)}_{t_1} = 0$. We argue as in the proof of (\ref
{e:monotonie}), and assume that $s_0 \in[t_1, t ]$ is such that
$r^{(1)}_{s_0} = r_{s_0}$. Since, on $C$, the jumps that affect both
fronts between time $s_0$ and time $t$ are exactly the same, one must
have that $r^{(1)}_{s} = r_{s}$ for all $s \in[s_0, t ]$. This proves
the claim.

Now by (\ref{e:comp-reste-positif}), we have that, on $C$, $r^{(1)}_t
\leq r_t$,
so that
$\P_{\nu_{\mathscr{C}, +}}(r_t \leq\beta t )$ is bounded above by
$\P
_{\nu_{\mathscr{C}, +}}(r^{(1)}_t \leq\beta t ) + \P_{\nu
_{\mathscr
{C}, +}}(C^c)$.
Thanks to (\ref{e:au-dessus-de-zero}), we have that $\P_{\nu
_{\mathscr
{C}, +}}  ( C^c  ) \leq a_1 t^{-a_2 \cdot\mathscr{C}}$.
On the other hand, by (\ref{e:compare-final}),
the distribution of $r^{(1)}_t$ with respect to $\P_{\nu_{\mathscr{C},
+}}$ stochastically dominates that of $r_{t_0}$ with respect to
$\P_{\nu_3}$. Using Proposition~\ref{p:ballistique-dessous-KS} with
$K:=\mathscr{C}$, and the fact that $\beta$ is chosen such that
$\beta< (1/3)C_2(\rho/4)$, we have that
\[
\P_{\nu_{\mathscr{C}, +}} \bigl(r^{(1)}_t \leq\beta t \bigr) \leq
\P_{\nu
_3}(r_{t_0} \leq\beta t) \leq\P_{\nu_3}
\bigl(r_{t_0} \leq C_2(\rho/4) t_0 \bigr) \leq
\Clconstcccnulis t_0^{-\mathscr{C}}.
\]
We have thus proved a bound of the desired form for $\P_{\nu
_{\mathscr
{C}, +}}(r_t \leq\beta t )$. Going from such a bound to a similar one
for $\P_{\nu_{\mathscr{C},+}}( \exists s \geq t;   r_s \leq\alpha s
)$ is easy, and we omit the details (see, e.g., the proof of Corollary
\ref{c:prob-super-alpha} for related ideas).

\subsubsection{Proof of Corollary \protect\ref{c:prob-super-alpha}}

We now prove Corollary~\ref{c:prob-super-alpha}.
Let $G$ denote the event that $t=0$ is a backward super-$\alpha$ time,
and observe that $\P_{\nu_{\mathscr{C},+}}(G)>0$,
using Lemma~\ref{l:controle-proba-atteinte}
and the symmetry of the distribution of our random walks. For $n \geq
1$, define $A_{n,1}:= \{ r_n \geq\beta n \}$ and
let $A_{n,2}$ denote the event that the particle at the front at time
$n$ with the smallest label
remains above level $\alpha(n+1)$ during the time-interval $[n, n+1]$.
For $k \geq1$, introduce the event
$A^{(k)}:= \{ \mbox{$r_t \geq\alpha t $ for all $t \geq k$} \}$,
and note that
$ \bigcap_{n \geq k} (A_{n,1} \cap A_{n,2}) \subset A^{(k)}$.
By Proposition~\ref{p:ballistique-dessous}, one has that
$\P_{\nu_{\mathscr{C}, +}}(A_{n,1}^c) \leq\Clconstcccvienas n^{-\Clconstcccdu
\cdot
\mathscr{C}}$.
Then, using, for example, a variance bound for the random walk, one has that,
for large enough $n$,
$\P_{\nu_{\mathscr{C}, +}}(A_{n,1} \cap A_{n,2}^c) \leq\Clconstddcvienas
n^{-2}$, for some constant $\Clconstddcvienas>0$.
As a consequence, we have that, for all large enough $\mathscr{C}$,
there exists $k \geq1$ such that
$\sum_{n \geq k} \P_{\nu_{\mathscr{C}, +} } ( A_{n,1}^c \cup A_{n,2}^c
) < \P_{\nu_{\mathscr{C}, +} }(G)$.
We thus have that
$ \P_{\nu_{\mathscr{C}, +} }  (G \cap\bigcap_{n \geq k} (A_{n,1}
\cap A_{n,2})  ) > 0$,
whence the fact that $ \P_{\nu_{\mathscr{C}, +} } ( G \cap A^{(k)} )
> 0$.
Let $U_0$ denote the largest label of a particle path $(W,u)$ such that
$W_0=0$ (if there is no such particle path, we set $U_0:=0$).
We deduce from the fact that $ \P_{\nu_{\mathscr{C}, +} } (G \cap
A^{(k)} ) > 0 $ the existence of a $u_0<1$ such that
%
\begin{equation}
\label{e:grimpe-tronque} \P_{\nu_{\mathscr{C}, +} } \bigl(G \cap A^{(k)} \cap\{
U_0 \leq u_0 \} \bigr) > 0.
\end{equation}
Now let $\Psi_0$ denote the subset of $\Psi$ obtained by removing all
particle paths $(W,u)$ such that
$W_0=0$ and $u > u_0$. We deduce from (\ref{e:grimpe-tronque}) that
\[
\P_{\nu_{+} } \bigl(G(\Psi_0) \cap A^{(k)}(
\Psi_0) \cap \bigl\{ \bigl\llvert X_0(\Psi
_0) \bigr\rrvert \geq\mathscr{C} \bigr\} \bigr) > 0,
\]
with the convention that, for $D \in\F$, $D(\Psi_0)$ denotes the event
that $\un_D(\Psi_0)=1$.
Now introduce the event $A'$ that:
\begin{itemize}
\item there exists a particle path $(W,u)$ such that $u > u_0$ and
$W_s=0$ for $s \in[0,\alpha^{-1}]$, and another particle path $(W,u)$
such that $u > u_0$,
$W_s \geq\lfloor\alpha s \rfloor $ for all $s \in[0,k]$;
\item every particle path $(W,u)$ such that $W_0=0$ and $u>u_0$
satisfies $W_s > \alpha s$ for $s < 0$.
\end{itemize}
One clearly has that $ \P_{\nu_{+} }(A')>0$, and that the two events
$A'$ and $G (\Psi_0) \cap A^{(k)}(\Psi_0) \cap\{ \llvert  X_0(\Psi_0)\rrvert   \geq
\mathscr{C} \}$ are independent with respect to $ \P_{\nu_{+} }$, and
[using (\ref{e:monotonie})], that $A^{(k)}(\Psi_0) \cap A' $ implies
that $0$ is a forward super-$\alpha$time for $B_0$.

\subsection{Conditional distribution of \texorpdfstring{$\pi_{r_{S_n},S_n}(B_{S_n})$}{pi{r{Sn},Sn}(B{Sn})}}\label{ss:cond-dist-Sn}

We now give an analogue of Proposition~\ref{p:conditionnement-1}
describing the conditional distribution of $\pi_{r_{S_n},
S_n}(B_{S_n})$ given
$ \G^R_{S_n}$. The reason why the result is not exactly the same as in
Proposition~\ref{p:conditionnement-1} is that the definition of $S_n$
involves additional conditions on trajectories in $B_{S_n}$, beyond
those saying that the trajectory of a particle that is blue at time
$S_n$ has to avoid the front between time $0$ and $S_n$.

Indeed, one first (mild) condition is that the number of particles
located at site $r_{S_n}$ at time $S_n$ has to be $\geq\mathscr{C}$.
More importantly, for each $1 \leq k < n$, we record whether $S_k$ is
or is not a backward super-$\alpha$ time, and this leads to the
following two types of conditions on $B_{S_n}$:
\begin{longlist}[(a)]
\item[(a)] when $S_k$ is a backward super-$\alpha$ time, all the
trajectories in $B_{S_n}$ have to lie above a space--time half-line of
slope $\alpha$ extending from $(r_{S_k},S_k)$ in the past direction;
\item[(b)] when $S_k$ is not a backward super-$\alpha$ time, all the
trajectories $(W,u)$ in $B_{S_n}$ such that $(W_{S_k},u)$ is less than
$(W^{(k)}_{S_k}, u^{(k)})$ with respect to the lexicographical order,
have to lie above a space--time half-line of slope $\alpha$ extending
from $(r_{S_k},S_k)$ in the past direction, where $(W^{(k)}, u^{(k)})$
denotes the witness trajectory contained in~$\Upsilon_k$.
\end{longlist}
Note that we have defined $S_n$ in such a way that the witness
trajectories contained in those $\Upsilon_{k}$ for which $1 \leq k
\leq
n-1$ and $S_k$ is not a backward super-$\alpha$ time, are included in
$R_{S_n}$, so that the above conditions can indeed be expressed using
only the information available in $\G^R_{S_n}$. Moreover, one checks
that, provided that $S_n$ is a backward super- and sub-$\alpha$ time,
the additional conditions (a) and (b) are automatically satisfied.
Since $S_n$ is by definition always a backward sub-$\alpha$ time,
conditions (a) and (b) will in fact be satisfied as soon as $S_n$ is a
backward super-$\alpha$ time.

Adapting the arguments leading to Proposition~\ref
{p:conditionnement-1}, we thus obtain the following result.
%
\begin{prop}\label{p:conditionnement-S}
For any bounded measurable map $F:  \Omega\to\R$, and all $n
\geq
1$, one has that, on $\{ S_n < +\infty\}$,
\[
\E_{\nu} \bigl(F \bigl(\pi_{r_{S_n}, S_n}(B_{S_n}) \bigr)
\mid \G^R_{S_n} \bigr) = \eta \bigl(F, Q^{(n)}
\bigr) \qquad\mbox{a.s.},
\]
where $Q^{(n)}$ is a $\G^R_{S_n}$-measurable random variable, and
where
%
\begin{equation}
\label{e:def-eta}\eta(F, \mathfrak{q}):= \E_{\nu
} \bigl(F(B_0)
\mid \mathfrak{G}( \mathfrak{q}) \cap\{ \Xi_0 = 1 \} \bigr),
\end{equation}
where
$\mathfrak{G}(\mathfrak{q})$ is an event that serves to encode
conditions \textup{(a)} and \textup{(b)} discussed above on $\pi_{r_{S_n},
S_n}(B_{S_n})$, and has the property that
%
\begin{equation}
\label{e:bouc} G \subset\mathfrak{G}(\mathfrak{q}),
\end{equation}
where $G:=\{\mbox{$t=0$ is a backward super-$\alpha$ time} \}$.
\end{prop}

Although exact, the description of the distribution of $\pi_{r_{S_n},
S_n}(B_{S_n})$ given by Proposition~\ref{p:conditionnement-S} may not
be very handy to prove explicit bounds. Fortunately, we have observed
that the complicated explicit conditions on $B_{S_n}$ are automatically
satisfied if one assumes that $S_n$ is also a backward super-$\alpha$
time. As a result, the following comparison between the (conditional)
distribution of $\pi_{r_{S_n}, S_n}(B_{S_n})$ and the much simpler
distribution $\P_{\nu_{\mathscr{C}, +}}$, can be derived from
Proposition~\ref{p:conditionnement-S}.

%
\begin{coroll}\label{c:comparaison-conditionnement}
For any nonnegative bounded measurable map $F:  \Omega\to\R$,
the following bound holds for all $n \geq1$, on $\{ S_n < +\infty\}$:
\[
\E_{\nu} \bigl(F \bigl(\pi_{r_{S_n}, S_n}(B_{S_n}) \bigr)
\mid \G^R_{S_n} \bigr) \leq \Cdev
\E_{\nu_{\mathscr{C}, +}}( F ) \qquad\mbox{a.s.},
\]
where $\Cdev$ is a positive constant depending on $\mathscr{C}$.
\end{coroll}

\begin{pf}
From (\ref{e:def-eta}) and (\ref{e:bouc}), we have that, for
nonnegative $F$,
\[
\eta(F, \mathfrak{q}) \leq\frac{\E_{\nu}(F(B_0) \un_{\Xi_0=1})
}{\P
_{\nu}( G \cap\{ \Xi_0 =1\} )} \leq\Cdev \E_{\nu_{\mathscr{C},+}}(F),
\]
with
\[
\Cdev:= \frac{ \P_{\nu}( \Xi_0 =1) }{ \P_{\nu}( G \cap\{
\Xi_0
=1\} ) },
\]
using Lemma~\ref{l:controle-proba-atteinte} to establish that $ \P
_{\nu
}( G \cap\{ \Xi_0 =1\} ) > 0$.
The conclusion now follows from Proposition~\ref{p:conditionnement-S}.
\end{pf}

\subsection{Tail bounds}\label{ss:tails}

We are now ready to prove the tail bounds that are necessary to control
the regeneration times. Let us describe how the proof is organized. On
the whole, the tools at our disposal are the following:
\begin{itemize}
\item comparison of the distribution of $\pi_{r_{S_n}, S_n}(B_{S_n})$
with $\nu_{\mathscr{C},+}$,
\item bounds on the ballistic behavior of the front with respect to $\P
_{\nu_{\mathscr{C},+}}$,
\item bounds on the hitting time/probability of a straight line of
slope $\alpha$ by a system of random walks, controlled by $\phi
_{\theta}$.
\end{itemize}

A first step, using these tools, consists in proving tail bounds for
the random variables $S_{n+1}-S_n$ and $r_{S_{n+1}}-r_{S_n}$ on the
event that $D_n<+\infty$, conditional upon~$\F^R_{S_n}$. Basically, the
idea is that, due to the ballisticity of the front, the failure of
$S_n$ to be an $\alpha$-separation time can be detected by looking at
trajectories ``around'' $(S_n,r_{S_n})$, while, for the same reason, the
post-$D_n$ conditions that characterize $S'_{n+1}$ and $S_{n+1}$ have
to be satisfied within a ``short'' interval of time. The bounds are
stated in Proposition~\ref{p:tbcof}, where the time intervals
$[S_n,S'_n]$ and $[S'_n,S_{n+1}]$ are dealt with separately.

Apart from well-controlled deterministic quantities, these bounds
involve the random quantity
$\MM_n$, defined for all $n \geq1$, on the event $\{ S_n < +\infty\}
$, by \index{MMn@{$\MM_n$}}
%
\begin{equation}
\label{e:def-M} \MM_n:= \sum_{(W,u) \in R^*_{S_n} }
e^{-\theta(r_{S_n} - W_{S_n})},
\end{equation}
where $R^*_{S_n}$ is defined in Section~\ref{ss:def-SD}. Broadly
speaking, $\MM_n$ measures the accumulation of particles below the
front position $r_{S_n}$, and the appearance of such a quantity in the
bounds comes from the necessity to control the forward sub-$\alpha$
time property at time $S_n$, which is done with the help of the
function $\phi_{\theta}$ applied to the particles whose positions are
$\leq r_{S_n}$ at time $S_n$.

To get rid of this random term and obtain deterministic tail bounds, we
need to control the evolution of $\MM_n$. Thus, the second step in the
proof consists in establishing an affine induction inequality
(Proposition~\ref{p:borne-esperance}) for the conditional expectation
of $\MM_n$, whose coefficient can be made $<1$ for a suitable choice of
the parameter $L$. The third, easier step, consists in proving tail
bounds in the case $n=1$ (i.e., for $S_1$, $\MM_1$, etc.), stated in
Proposition~\ref{p:bornes-initiales}. With this step completed, the
previous results can be combined to prove a uniform bound (Proposition
\ref{p:bornitude}) on the conditional expectation of $\MM_n$ given that
$D_{n-1}<+\infty$, which is the missing ingredient needed to obtain the
suitable deterministic tail bounds on the renewal structure that imply
finiteness of the second moments.

Before we enter the various steps of the proof, let us first note that,
thanks to Proposition~\ref{p:ballistique-dessous} and to the fact that
$\alpha< \beta$, we know that, for $n \geq0$, when $D_n < +\infty$,
one almost surely has that $S_{n+1} < +\infty$.

\subsubsection{Step 1: Tail bounds conditional on $\F^R_{S_n}$}

The following proposition lists the various bounds we have.

%
\begin{prop}\label{p:tbcof}
For all $n \geq1$, for all $t>0$ and $K > 0$, following bounds hold on
$\{ S_n < +\infty\}$:
\begin{eqnarray}
&&\P_{\nu} \bigl( S'_{n+1} -
S_n \geq t, D_{n}<+\infty\mid \F^R_{S_n}
\bigr)
\nonumber\\[-8pt]\label{e:queue-S'-S}  \\[-8pt]\nonumber
&&\qquad  \leq e^{\theta} \MM_{n} e^{-\Clconstccseptyn t} +
c_{11} t^{- c_{12} \mathscr{C}} \qquad\mbox{a.s.},
\\
&&\P_{\nu} \bigl( r_{S'_{n+1}} - r_{S_n}
\geq K, D_{n}<+\infty\mid \F^R_{S_n} \bigr)
\nonumber\\[-8pt]\label{e:queue-r-S'-S}  \\[-8pt]\nonumber
&&\qquad \leq \MM_{n} e^{-c_{13} K} + c_{14} K^{- c_{15} \mathscr{C}} \qquad\mbox{a.s.},
\\
\label{e:queue-S-S'}
&&\P_{\nu} \bigl( S_{n+1} -
S'_{n+1} \geq t, D_{n}<+\infty\mid
\F^R_{S'_{n+1}} \bigr) \leq c_{16} t^{- c_{17} \mathscr{C}} \qquad\mbox{a.s.},
\\
\label{e:queue-r-S-S'}
&&\P_{\nu} \bigl( r_{S_{n+1}} - r_{S'_{n+1}}
\geq K, D_{n}<+\infty\mid \F^R_{S'_{n+1}} \bigr)
\leq c_{18} K^{-
c_{19} \mathscr{C}} \qquad\mbox{a.s.},
\end{eqnarray}
where $\Clconstccseptyn,\ldots,c_{19}$ are strictly positive constants, with
$c_{11}, c_{14}$ depending on $\mathscr{C}$, and $c_{16},c_{18}$ depending on $\mathscr{C}$ and $L$.
\end{prop}
The proof of Proposition~\ref{p:tbcof} relies on controlling the
numbers of $\alpha$-crossings in the relevant time intervals, so for $n
\geq1$, let $\NN_n$ and $\NN'_n$, respectively, denote the number of
$(S_n, \alpha)$-crossings contained in the time-interval $[S_n,
S'_{n+1}]$, and the number of $(S'_{n+1}, \alpha)$-crossing times
contained in the time-interval $[S'_{n+1}, S_{n+1}]$. The key estimates
on these variables are given in the following lemma.
%
\begin{lemma}\label{l:controle-SS'}
One has the following bounds: for all $n \geq1$, for all $K > 0$, on
$\{ S_n < +\infty\}$,
\begin{eqnarray}
\label{e:borne-NN} \qquad\P_{\nu} \bigl( \NN_n \geq K,
D_{n}<+\infty\mid \F^R_{S_n} \bigr) &\leq&
e^{\theta} \MM_{n} e^{-c_{20} K} + c_{21}
K^{- c_{22} \mathscr{C}} \qquad\mbox{a.s.},
\\
\label{e:borne-NN'} \P_{\nu} \bigl( \NN'_n \geq
K, D_{n}<+\infty\mid \F^R_{S'_n} \bigr) &\leq&
c_{23} K^{- c_{24} \mathscr{C}} \qquad\mbox{a.s.},
\end{eqnarray}
where
$c_{20},\ldots,c_{24}$ are strictly positive constants, with
$c_{21}$ depending on $\mathscr{C}$ and $c_{21}$ depending on
$\mathscr{C}$ and $L$.
\end{lemma}
\begin{pf*}{Proof of Lemma~\ref{l:controle-SS'}}
Let $S''_{n+1}$ denote the infimum of the $t > D_{n}$ such that:
\begin{itemize}
\item$t$ is a backward sub-$\alpha$ time;
\item$\Upsilon_{n} \subset R_t$.
\end{itemize}

Let $\NN_n^{(1)}$ and $\NN_n^{(2)}$ denote, respectively, the numbers of
$(S_n, \alpha)$-crossings contained in the time-interval $[S_n,
S''_{n+1}[$ and in the time-interval
$[S''_{n+1}, S'_{n+1}]$, so that
%
\begin{equation}
\label{e:addition-N}\NN_n = \NN_n^{(1)} +
\NN_n^{(2)}.
\end{equation}
Our first claim is that there exists a constant $d_1<1$,
depending on $\mathscr{C}$, such that, for all $\ell\geq1$,
%
\begin{equation}
\label{e:queue-plus-de-C} \P_{\nu} \bigl( \NN_n^{(2)} \geq
\ell, D_{n}<+\infty\mid \G^R_{S_n} \bigr) \leq
d_1^{\ell} \qquad\mbox{a.s.}
\end{equation}
Assume that $D_n < +\infty$, and
denote by $\tau_1, \tau_2,\ldots$ the successive backward
sub-$\alpha$
times posterior to $S''_{n+1}$ (with $\tau_1:= S''_{n+1}$), and let
$J:= \inf\{ j \geq1;   \Xi_{\tau_j} = 1 \}$ (remember that $\Xi
_t=1$ means that there are at least $\mathscr{C}$ particles located at
site $r_t$ in $B_t$). By definition, we have $S'_{n+1} = \tau_J$.
Since $S_n$ is a backward sub-$\alpha$ time, any $(S_n,\alpha
)$-crossing in $[S''_{n+1}, S'_{n+1}]$ is a backward sub-$\alpha$ time,
so we have
\begin{equation}
\label{e:compte-J}\NN_n^{(2)} \leq J.
\end{equation}
Now using an argument similar to the one leading to Proposition~\ref
{p:conditionnement-S}, we have that, for all $i \geq1$, on $\{ D_n <
+\infty\}$, the distribution of $\pi_{r_{\tau_i}, \tau_i}(B_{\tau
_i})$ conditional upon $\G^R_{\tau_i}$ is that of $B_0$ conditioned
upon an event containing $G$, so that, on $\{ D_n < +\infty\}$, one
has the bound
\[
\P_{\nu} \bigl( \Xi_{\tau_i} = 1 \mid \G^R_{\tau_i}
\bigr) \geq\P_{\nu} \bigl(\{ \Xi_0 = 1 \} \cap G \bigr)
\qquad\mbox{a.s.}
\]
Since for all $i \geq2$, the random variables $\Xi_{\tau_1},\ldots,
\Xi
_{\tau_{i-1}}$ are measurable with respect to
$\G^R_{\tau_i}$, we deduce that, on $\{ D_n < +\infty\}$,
%
\begin{equation}
\label{e:geom}\P_{\nu} \bigl( J \geq\ell\mid \G^R_{S''_{n+1}}
\bigr) \leq \bigl(1- \P_{\nu} \bigl(\{ \Xi_0 = 1 \} \cap G
\bigr) \bigr)^{\ell} \qquad\mbox{a.s.}
\end{equation}
Combining (\ref{e:compte-J}) and (\ref{e:geom}), we deduce (\ref
{e:queue-plus-de-C}), using also the fact that\footnote{Note that this
property is not obvious. It is a consequence that it is enough to look
at trajectories in $R_{S_n}$ to check the backward super-$\alpha$ time
property for $S_{j}$ where $j \leq n-1$. See the discussion before
Proposition~\ref{p:conditionnement-S}.} $\G^R_{S_n} \subset\G
^R_{S''_{n+1}}$ since $S_n \leq S''_{n+1}$ and
$S_n$ is $\G^R_{S''_{n+1}}$-measurable.

Now consider the event $\NN_n^{(1)} > \ell$. Start with the case where
$S_n$ is not a backward super-$\alpha$ time, and call \index
{Hn@{$H_n$}} $H_n$ the corresponding event. We first bound the
probability that
$W_{S_n}^{(n)} > r_{S_n} + \ell/2$. From Lemma~\ref{l:atteinte-droite},
a random walk starting at $x \geq0$ at time zero has a probability
bounded above by $e^{-\theta x}$ to ever cross at a negative time the
half-line of slope $\alpha$ starting at $(0,0)$.
Using Corollary~\ref{c:comparaison-conditionnement} and the union bound
over all the particle paths in $B_{S_n}$, we deduce that
\[
\P_{\nu} \bigl( H_n, W_{S_n}^{(n)} >
r_{S_n} + \ell/2 \mid \G^R_{S_n} \bigr) \leq \Cdev \sum_{x > \ell/2} e^{-\theta x} \rho.
\]
We deduce that
%
\begin{equation}
\label{e:H-un-bout} \P_{\nu} \bigl( H_n, W_{S_n}^{(n)}
> r_{S_n} + \ell/2 \mid \G^R_{S_n} \bigr) \leq
d_2 e^{-d_3
\ell},
\end{equation}
where $d_2$ and $d_3$ are strictly positive constants, with
$d_2,d_3$ depending on $\mathscr{C}$.
On the other hand, assume that $W_{S_n}^{(n)} \leq r_{S_n} + \ell/2$.
Assume also that $\NN_n^{(1)} > \ell$, and let $t$ denote the time of
the $\ell$th $(S_n,\alpha)$-crossing posterior to $S_n$. By definition
of $\NN_n^{(1)}$, we must have that $(W^{(n)}, u^{(n)}) \notin R_t$,
whence $W^{(n)}_t \geq r_t \geq r_{S_n}+\ell+\alpha(t-S_n)$. Since
$W_{S_n}^{(n)} \leq r_{S_n} + \ell/2$, this implies that $W^{(n)}_t
\geq W^{(n)}_{S_n} + \ell/2+\alpha(t-S_n)$. Using again Corollary
\ref
{c:comparaison-conditionnement}, Lemma~\ref{l:atteinte-droite} and the
union bound, we deduce that
\[
\P_{\nu} \bigl(\NN_n^{(1)} > \ell,
H_n, W_{S_n}^{(n)} \leq r_{S_n} + \ell
/2 \mid \G^R_{S_n} \bigr) \leq\Cdev \biggl(
e^{-\theta\ell/2} \rho_{\mathscr{C}} + \sum_{1 \leq x \leq\ell/2}
e^{-\theta\ell/2} \rho \biggr),
\]
where
$\rho_{\mathscr{C}} $ denotes the expected value of a Poisson random
variable of parameter $\rho$ conditioned upon being
$\geq\mathscr{C}$ (we have to use $\rho_{\mathscr{C}} $ since, under
$\nu_{\mathscr{C},+}$, the number of particles at the origin has a
Poisson distribution conditioned by taking a value $\geq\mathscr{C}$).
We deduce that there exists a strictly positive constant $d_4$ depending on $\mathscr{C}$ and a strictly positive constant
$d_5$ such that
%
\begin{equation}
\label{e:H-un-autre-bout} \P_{\nu} \bigl( \NN_n^{(1)} >
\ell, H_n, W_{S_n}^{(n)} \leq r_{S_n} +
\ell/2 \mid \G^R_{S_n} \bigr) \leq d_4 e^{-d_5 \ell}.
\end{equation}

Now consider the case where $S_n$ is a backward super-$\alpha$ time. In
this case, $\Upsilon= \varnothing$ and, by definition of $S''_{n+1}$,
$\NN_n^{(1)}$ is also the number of $(S_n, \alpha)$-crossings contained
in the time-interval $[S_n, D_n]$. Introduce $t_0:= \ell/C'_1(\rho)$
(remember that $C'_1(\rho)$ is defined in Proposition~\ref
{p:ballistique-dessus-remanent-KS}), assuming that $\ell$ is large
enough so that $t_0 > \alpha^{-1}$, and consider the cases
$D_n-S_n > t_0$ and $D_n - S_n \leq t_0$ separately. Assume first that
$D_n - S_n \leq t_0$, and let $t$ denote the time of the $\ell$th
$(S_n,\alpha)$-crossing posterior to~$S_n$. The fact that $\NN_n^{(1)} >
\ell$ implies that $t < D_n$,
while $r_{t} \geq r_{S_n} + \ell$. Moreover, since $t < D_n$, $r_{t}$
is in fact equal to $r_{S_n}+r_{t-S_n}(\pi_{r_{S_n}, S_n}(B_{S_n}))$
since, by definition, particles in $R_{S_n}$ cannot influence the front
between time $S_n$ and time $D_n$. As a consequence, $r_{t-S_n}(\pi
_{r_{S_n}, S_n}(B_{S_n})) \geq\ell$, while $t-S_n \leq t_0$. Using
Corollary~\ref{c:comparaison-conditionnement} and Proposition~\ref
{p:ballistique-dessus-KS-extension}, we deduce that
%
\begin{equation}
\label{e:Hc-un-bout} \P_{\nu} \bigl( \NN_n^{(1)} \geq
\ell, H_n^c, D_n-S_n \leq
t_0 \mid \G^R_{S_n} \bigr) \leq\Clconstcccseptyni
e^{-\Clconstcccdesimt  t_0}.
\end{equation}

On the other hand, using again the fact that particles in $R_{S_n}$
cannot influence the front prior between time $S_n$ and time $D_n$, we
see that, if $D_n - S_n > t_0$, at least one of the three following
events must occur, according to which of the five conditions defining
$D_n$ corresponds to the smallest time (note that our assumption that
$t_0 > \alpha^{-1}$ rules out $(2)$ and $(4)$): for some $t \geq S_n +
t_0$, $r_{t}( \pi_{r_{S_n}, S_n}(B_{S_n}) ) < \lfloor\alpha
{(t-S_n)} \rfloor $, or
there exists a particle path $(W,u) \in R_{S_n}$ such that $W_{S_n}
\leq r_{S_n}-1$ and a $t \geq t_0$ such that
$W_{S_n+t} \geq r_{S_n}-1+ \alpha(t-S_n)$, or there exists a $t \geq
t_0$ such that
$W^{*n}_t > r_{S_n}-1+\alpha t$, while $W^{*n}_{S_n+\alpha^{-1}}=r_{S_n}$.
Using Corollary~\ref{c:comparaison-conditionnement}, Proposition~\ref
{p:ballistique-dessous}, Lemma~\ref{l:controle-proba-atteinte} and
Lemma~\ref{l:atteinte-droite}, and the strong Markov property\footnote
{The Markov property of the dynamics holds with respect to the
filtration $(\F_t)_{t \geq0}$, not $(\F^R_t)_{t \geq0}$. Here, we use
the fact that $\F^R_{S_n} \subset\F_{S_n}$, and also that $\F^R_{S_n}
\subset\G^R_{S_n}$.} at time $S_n$ and $S_n + \alpha^{-1}$,
we deduce by the union bound that
%
\begin{eqnarray}
\label{e:Hc-deuxieme-bout}
&& \P_{\nu} \bigl( \NN_n^{(1)} \geq
\ell, H_n^c, t_0 <D_n-S_n
< +\infty\mid \F ^R_{S_n} \bigr)
\nonumber\\[-8pt]\\[-8pt]\nonumber
&&\qquad  \leq\Cdev \Clconstcccvienas
t^{-\Clconstcccdu \cdot\mathscr{C}} + e^{\theta} \MM_{n} e^{-\mu t_0}
+ e^{-\mu(t_0 - \alpha^{-1})}.
\end{eqnarray}

Putting together (\ref{e:addition-N}), (\ref{e:queue-plus-de-C}),
(\ref{e:H-un-bout}), (\ref{e:H-un-autre-bout}), (\ref{e:Hc-un-bout}) and
(\ref{e:Hc-deuxieme-bout}), we deduce the first part of the lemma, that
is, the bound (\ref{e:borne-NN}).

To prove (\ref{e:borne-NN'}), we define $S^{\prime\prime\prime}_{n+1}$ as the infimum
of the $t>S'_{n+1}$ such that $t$ is a backward sub-$\alpha$ time and
$]S'_{n+1}, t[$ contains at least $L$ $(S'_{n+1},\alpha)$-crossing
times, and let $\tau'_1,\tau'_2,\ldots$ denote the successive backward
sub-$\alpha$ times posterior to $S^{\prime\prime\prime}_{n+1}$ (with $\tau'_1:=
S^{\prime\prime\prime}_{n+1})$, and let $I:= \inf\{ i \geq1,  \Xi_{\tau'_i} =1 \}
$. We have by definition that $\NN_n' \leq L + I$. Arguing exactly
as in the proof of (\ref{e:borne-NN}), we can prove a bound of the form
%
\begin{equation}
\label{e:queue-plus-de-C-2} \P_{\nu} \bigl( I \geq\ell, D_{n}<+\infty
\mid \G^R_{S'_{n+1}} \bigr) \leq d_6^{\ell}
\qquad\mbox{a.s.},
\end{equation}
where $d_6<1$ is a constant depending on $\mathscr{C}$, which
leads to the desired bound on the tail of $\NN_n'$.
\end{pf*}

We now give the proof of Proposition~\ref{p:tbcof}, which heavily
relies on Lemma~\ref{l:controle-SS'} we have just proved.

\begin{pf*}{Proof of Proposition~\ref{p:tbcof}}
We start with the proof of (\ref{e:queue-S'-S}). Assume that
$S'_{n+1}-S_n \geq t$. If $r_{S_n+t} \geq r_{S_n} + \beta t$, we deduce
that there exist at least $\lfloor\frac{(\beta-\alpha)t}{2}
\rfloor $ distinct
$(S_n,\alpha)$-crossing times in $[S_n,S_n+t]$, whence the fact that
$\NN_n \geq\lfloor\frac{(\beta-\alpha)t}{2} \rfloor $. On the
other hand,
using (\ref{e:monotonie}),
Proposition~\ref{p:ballistique-dessous} and Corollary~\ref
{c:comparaison-conditionnement}, we see that
\[
\P_{\nu} \bigl(r_{S_n+t} - r_{S_n} \leq\beta t,
D_{n}<+\infty\mid \F ^R_{S_n} \bigr) \leq
\Cdev \Clconstcccvienas t^{-\Clconstcccdu \cdot\mathscr{C}} \qquad\mbox{a.s.}
\]
The conclusion now follows from (\ref{e:borne-NN}).
To prove (\ref{e:queue-r-S'-S}), we note that, by definition of
$\alpha
$-crossing times, we have
\[
r_{S'_{n+1}} \leq r_{S_n}+ \alpha \bigl(S'_{n+1}-S_n
\bigr)+\NN_n+1.
\]
The result then follows from combining (\ref{e:queue-S'-S}) and (\ref{e:borne-NN}). The proof of (\ref{e:queue-S-S'}) is similar to
that of (\ref{e:queue-S'-S}), using (\ref{e:borne-NN'}) instead of
(\ref{e:borne-NN}),
Proposition~\ref{p:ballistique-dessous} and an analog of Corollary~\ref
{c:comparaison-conditionnement} for $S'_{n+1}$. The proof of (\ref
{e:queue-r-S-S'}) goes as the proof of (\ref{e:queue-r-S'-S}), building
on (\ref{e:queue-S-S'}) and~(\ref{e:borne-NN'}).
\end{pf*}

\subsubsection{Step 2: Evolution of $\MM_n$}

The main estimate we prove is the following affine induction inequality
for $\MM_n$.
%
\begin{prop}\label{p:borne-esperance}
For all $n \geq1$, and all large enough $\mathscr{C}$, one has the
following bounds:
\[
\E_{\nu} \bigl( \MM_{n+1} \un(D_n < +\infty)
\mid \F^R_{S_n} \bigr) \leq c_{25}
e^{-\theta L} \MM_{n} + c_{26},
\]
where $ c_{25}$ is a strictly positive constant depending on
$\mathscr{C}$, and $ c_{26}$ is a strictly positive constant
depending on $\mathscr{C}$ and $L$.
\end{prop}

The proof of~\ref{p:borne-esperance} involves a martingale inequality,
proved in Lemma~\ref{l:martingale-exp}, and quantitative bounds on
$\mathcal{L}_n^{(1)}:={}$number of particle paths in $B_{S_n} \cap
R_{S'_{n+1}}$ and $\mathcal{L}_n^{(2)}:={}$number of particle paths in
$B_{S'_{n+1}} \cap R_{S_{n+1}}$, derived in Lemma~\ref
{l:accumulation-particules} with the help of the results proved in Step~1.

%
\begin{lemma}\label{l:martingale-exp}
Consider $w \in\S_{\theta}$ such that there is at least one particle
at site $0$. Let $T$ be an $(\F_t)_{t \geq0}$ stopping time such that
$T$ is a backward sub-$\alpha$ time and $]0,T[$ contains a number of
$(0,\alpha)$-crossing times at least equal to $m \geq0$. Then one has
the following bound:
\[
\E_w \biggl( \sum_{(W,u) \in R_{0+}}
e^{-\theta(r_T-W_T)} \un(T < +\infty ) \biggr) \leq e^{-\theta m}
\phi_{\theta}(w).
\]
\end{lemma}

\begin{pf}
Consider $(W,u) \in R_{0+}$, and, for all $s \geq0$, let
\[
M_s:=
e^{\theta W_s - 2 (\cosh(\theta) -1)s}.
\]
Then $(M_s)_{s \geq0}$ is a c\`adl\`ag martingale.
Since $T$ is a stopping time, we have, for all finite $K>0$, that
%
\begin{equation}
\label{e:Doob-2} \E_w(M_{T \wedge K}) = \E_w(M_0)
= e^{\theta W_0}.
\end{equation}
Now we have that $\liminf_{K \to+\infty} M_{T \wedge K} \geq M_{T}
\un
(T<+\infty)$, so that, by Fatou's lemma and (\ref{e:Doob-2}),
%
\begin{equation}
\label{e:magic-martingale-3}\E_w \bigl( M_{T} \un (T<+\infty) \bigr)
\leq e^{\theta W_0}.
\end{equation}
Now, from our assumptions on $T$ and the fact that $r_0=0$, one has
that, on $\{ T < +\infty\}$, $r_T \geq\alpha T + m $.
Using the fact that, by (\ref{e:constraints-param}), $ 2 (\cosh
(\theta
) -1) \leq\alpha\theta$, we deduce that
\[
2 \bigl(\cosh(\theta) -1 \bigr) T \leq2 \bigl(\cosh(\theta) -1 \bigr) \biggl(
\frac
{r_T-m}{\alpha} \biggr) \leq\theta(r_T-m).
\]
Writing
\[
-\theta r_T + \theta W_T = -\theta r_T +
2 \bigl(\cosh(\theta) -1 \bigr)T - 2 \bigl(\cosh(\theta) -1 \bigr)T + \theta
W_T,
\]
we finally deduce that, on $\{ T < +\infty\}$,
\[
-\theta r_T + \theta W_T \leq-\theta m - 2 \bigl(
\cosh(\theta) -1 \bigr)T + \theta W_T.
\]
In view of (\ref{e:magic-martingale-3}), we deduce that
\[
\E_w \bigl( e^{-\theta(r_T - W_T) } \un(T<+\infty) \bigr) \leq
e^{-\theta m} \E _w \bigl( M_{T} \un(T<+\infty)
\bigr) \leq e^{-\theta m} e^{\theta W_0}.
\]
The result now follows from summing the above inequality over all
$(W,u) \in R_{0+}$.
\end{pf}

%
\begin{lemma}\label{l:accumulation-particules}
For all $n \geq1$, and all large enough $\mathscr{C}$, one has the
following bounds:
\begin{eqnarray}
\label{e:accu-1} \E_{\nu} \bigl( \mathcal{L}_n^{(1)}
\un (D_n < +\infty) \mid \F^R_{S_n} \bigr)
&\leq&c_{27} + c_{28} \MM_{n} \qquad\mbox{a.s.},
\\
\label{e:accu-2} \E_{\nu} \bigl( \mathcal{L}_n^{(2)}
\un(D_n < +\infty ) \mid \F ^R_{S_n} \bigr)
&\leq& c_{29} \qquad\mbox{a.s.},
\end{eqnarray}
where $c_{27}$ and $c_{28}$ are strictly positive constants
depending on $\mathscr{C}$ and
$c_{29}$ is a strictly positive constant depending on $\mathscr{C}$
and $L$.
\end{lemma}
\begin{pf}
We start with the proof of (\ref{e:accu-1}). Assume that $S'_{n+1}
\leq
S_n + t$ and that $r_{S'_{n+1}} \leq r_{S_n} + K$ for some $t, K>0$.
We can then bound $\mathcal{L}_n^{(1)}$ by counting the total number of
particle paths $(W,u)$ in $B_{S_n}$ for which there exists $s \in[S_n,
S_n+t]$ such that $W_s \in[r_{S_n}, r_{S_n}+K]$.
This number includes all the particle paths $(W,u)$ in $B_{S_n}$ such
that $W_{S_n} \in[r_{S_n}, r_{S_n}+K]$, plus the particle paths in
$B_{S_n}$ such that $W_{S_n} \geq r_{S_n}+K+1$ that hit level
$r_{S_n}+K$ during the time-interval $[S_n, S_n + t]$.
Assume that we start with $\P_{\nu_{\mathscr{C},+}}$, and let
$\mathcal
{K}_1$ denote the number of particle paths $(W,u)$ in $B_{0}$ such that
$W_{0} \in[0, K]$, while
$\mathcal{K}_2$ denotes the number of particle paths in $B_{0}$ such
that $W_{0} \geq r_{0}+K+1$ that hit level $K$ during the
time-interval $[0, t]$.
By standard properties of the Poisson distribution, we see that
$\mathcal{K}_2$ is a Poisson random variable with distribution $\rho g$,
where
\[
g:= \sum_{x \geq K+1} P_x \Bigl(\inf
_{s \in[0,t]} \zeta_s \leq K \Bigr).
\]
Using the reflection principle, we see that $g \leq g'$, where
\[
g' = 2\sum_{x \geq K+1} P_x(
\zeta_t \leq K).
\]
Now using translation invariance, we can rewrite
\[
g' = 2\sum_{ x \geq1 } P_x(
\zeta_t \leq0 ) = 2\sum_{ x \geq1 }
P_0( x+\zeta_t \leq0 ) = 2 E_0 \bigl( -
\zeta_t \un(\zeta_t \leq-1) \bigr).
\]
Using Schwarz's inequality, we deduce that
$g' \leq2 \sqrt{2t}$.
On the other hand, $\mathcal{K}_1$ is the sum of $\Xi_0$, whose
distribution is that of a Poisson random variable of parameter $\rho$
conditioned to be $\geq\mathscr{C}$, and of a Poisson random variable
of parameter $\rho K$, these two variables being independent, and
independent from $\mathcal{K}_2$.
Using Corollary~\ref{c:comparaison-conditionnement}, we deduce that,
for some strictly positive constant $d_1$ depending on
$\mathscr{C}$,
%
\begin{equation}
\label{e:borne-Poisson} \P_{\nu} \bigl( \mathcal{L}_n^{(1)}
\geq\ell, S'_{n+1} \leq S_n + t,
r_{S'_{n+1}} \leq r_{S_n} + K \mid \G ^R_{S_n}
\bigr) \leq d_1 a_{t,K}(\ell),
\end{equation}
where $a_{t,K}(\ell)$ denotes the probability for a Poisson random
variable with parameter $\rho(K+1 +2 \sqrt{2t})$ to be $\geq\ell$.
Now consider two strictly positive constants $b_1, b_2$ with $\rho b_1
< 1$.
Note that, for $K:= b_1 \ell$ and $t:=b_2 \ell$, one has, by standard
large deviations bounds for Poisson random variables (see, e.g., \cite
{DemZei}), that for all $\ell\geq1$,
%
\begin{equation}
\label{e:deviations-Poisson} a_{t,K}(\ell) \leq d_2 e^{- d_3 \ell},
\end{equation}
where $d_2, d_3$ are strictly positive constants.
Combining (\ref{e:queue-S'-S}) and (\ref{e:queue-r-S'-S}), we deduce
that, on $\{ S_n < +\infty\}$,
%
\begin{equation}
\label{e:borne-t-K} \P_{\nu_{\mathscr{C},+}} \bigl( \mathcal {V}_n(
\ell)^c, D_n < +\infty\mid \F^R_{S_n}
\bigr) \leq e^{\theta} \MM_{n} e^{-d_4 \ell} +
d_5 \ell^{- d_6
\mathscr{C}} \qquad\mbox{a.s.},
\end{equation}
with
\[
\mathcal{V}_n(\ell):= \bigl\{ S'_{n+1} \leq
S_n + b_2 \ell \bigr\} \cap\{ r_{S'_{n+1}} \leq
r_{S_n}+ b_1 \ell\},
\]
and where $d_4, d_5, d_6$ are strictly positive
constants, $d_5$ depending on $\mathscr{C}$.
Combining (\ref{e:borne-Poisson}), (\ref{e:deviations-Poisson}) and
(\ref{e:borne-t-K}), we deduce (\ref{e:accu-1}).

To prove (\ref{e:accu-2}), we use the same argument as in the proof of
(\ref{e:accu-1}), with (\ref{e:queue-S-S'}) and (\ref
{e:queue-r-S-S'}) replacing
(\ref{e:queue-S'-S}) and (\ref{e:queue-r-S'-S}), respectively.
\end{pf}

We are now ready to prove the affine induction inequality on $\MM_n$.
\begin{pf*}{Proof of Proposition~\ref{p:borne-esperance}}
Define
\[
\MM'_{n+1}:= \sum_{(W,u) \in R^*_{S'_{n+1}}}
\exp \bigl( -\theta (r_{S'_{n+1}} - W_{S'_{n+1}} ) \bigr),
\]
where $R^*_{S'_{n+1}}$ is defined as the set $R_{S'_{n+1}}$ from which
we remove the particle path that makes the front climb at time $S'_{n+1}$.
By definition, we have that $\MM'_{n+1} \leq\A^{(1)} + \A^{(2)}$,
with
\[
\A^{(1)}:= \sum_{(W,u) \in R_{S_n}} \exp \bigl( -\theta
(r_{S'_{n+1}} - W_{S'_{n+1}} ) \bigr)
\]
and
\[
\A^{(2)}:= \sum_{(W,u) \in B_{S_{n}} \cap R_{S'_{n+1}} } \exp \bigl( -
\theta(r_{S'_{n+1}} - W_{S'_{n+1}} ) \bigr).
\]
First, using the fact that for each $(W,u) \in R_{S'_{n+1}}$, one has
$W_{S'_{n+1}} \leq r_{S'_{n+1}}$, we have the bound
%
\begin{equation}
\label{e:crude-1} \A^{(2)} \leq\mathcal{L}_n^{(1)}.
\end{equation}
Now using Lemma~\ref{l:martingale-exp}, we deduce that
%
\begin{equation}
\label{e:mart-1} E_{\nu} \bigl( \A^{(1)} \un(D_n
< +\infty) \mid \F^R_{S_n} \bigr) \leq\MM_{n} +
1,
\end{equation}
where the $+1$ term comes from the fact that the definition of $\MM_n$
involves the particles in $R_{S_n}^*$, not $R_{S_n}$, so we have to add
the contribution to $\A^{(1)}$ of the particle path $(W^{*n}, u^{*n})$, which we bound by $1$.
Now we have that $\MM_{n+1} \leq\B^{(1)} + \B^{(2)}$,
with
\[
\B^{(1)}:= \sum_{(W,u) \in R_{S'_{n+1}}} \exp \bigl( -\theta
(r_{S_{n+1}} - W_{S_{n+1}} ) \bigr)
\]
and
\[
\B^{(2)}:= \sum_{(W,u) \in B_{S'_{n+1}} \cap R_{S_{n+1}} } \exp \bigl( -
\theta(r_{S_{n+1}} - W_{S_{n+1}} ) \bigr).
\]
As in (\ref{e:crude-1}), we have the bound
%
\begin{equation}
\label{e:crude-2} \B^{(2)} \leq\mathcal{L}_n^{(2)}.
\end{equation}
On the other hand, using
Lemma~\ref{l:martingale-exp}, we deduce that
%
\begin{equation}
\label{e:mart-2} E_{\nu} \bigl( \B^{(1)} \un(D_n
< +\infty) \mid \F^R_{S'_{n+1}} \bigr) \leq e^{-\theta L}
\MM'_{n+1} + 1.
\end{equation}
Combining (\ref{e:crude-1}), (\ref{e:mart-1}), (\ref{e:crude-2}),
(\ref{e:mart-2}) and using the fact that $\F^R_{S_n} \subset\F
^R_{S'_{n+1}}$, we deduce that, on $\{ S_n < +\infty\}$,
\begin{eqnarray*}
\E_{\nu} \bigl( \MM_{n+1} \un(D_n < +\infty)
\mid \F ^R_{S_n} \bigr) &\leq& e^{-\theta L}
\MM_{n} + e^{-\theta L} + 1
\\
&&{} + e^{-\theta L} \E _{\nu} \bigl( \mathcal{L}_n^{(1)}
\un(D_n < +\infty) \mid \F^R_{S_n} \bigr)
\\
&&{}  + \E _{\nu} \bigl( \mathcal{L}_n^{(2)}
\un(D_n < +\infty) \mid \F^R_{S_n} \bigr).
\end{eqnarray*}
The conclusion now follows from Lemma~\ref{l:accumulation-particules}.
\end{pf*}

\subsubsection{Step 3: Tail bounds for $n=1$}

So far, we have proved results dealing with the behavior of the system
during the time-interval $[S_n, S_{n+1}]$, for $n \geq1$. The case of
the interval
$[0,S_1]$ is a little bit different since it starts at time $D_0=0$,
where not all the properties of times $S_n,  n \geq1$ are met.
However, the distribution of $(R_0,B_0)$ is exactly known, and, in this
case, Proposition~\ref{p:ballistique-dessous-KS} directly yields the
estimates that we obtained by a combination of Proposition~\ref
{p:ballistique-dessous} and Corollary~\ref
{c:comparaison-conditionnement} in the case $[S_n, S_{n+1}]$. We merely
state the relevant results, whose proofs are similar (and simpler) than
those given in steps 1~and~2.

%
\begin{prop}\label{p:bornes-initiales}
One has the following bounds:
\begin{eqnarray}
\label{e:queue-S'-S-0} \P_{\nu} \bigl( S'_{1} \geq
t \bigr) &\leq& c_{30} t^{- c_{31} \mathscr{C}} \qquad\mbox{a.s.},
\\
\label{e:queue-S-S'-0} \P_{\nu} \bigl( S_{1} -
S'_{1} \geq t \bigr) &\leq& c_{32}
t^{- c_{33} \mathscr{C}} \qquad\mbox{a.s.},
\\
\label{e:borne-esperance-0}
\E_{\nu}( \MM_{1} ) &\leq& c_{34},
\end{eqnarray}
where $c_{30},\ldots, c_{34}$ are strictly positive constants,
with $c_{30}$ depending on $\mathscr{C}$,
and $c_{32}, c_{34}$ depending on $\mathscr{C}$ and $L$.
\end{prop}

\subsubsection{Conclusion}

We now put together the different pieces established in the previous
steps. The first result is a uniform bound on the conditional
expectation of $\MM_n$ given that the first $n-1$ attempts at obtaining
an $\alpha$-separation time have failed.

%
\begin{prop}\label{p:bornitude}
For all large enough $\mathscr{C}$, and all large enough $L$ (depending
on $\mathscr{C}$), there exists $c_{35}<+\infty$, depending
on $\mathscr{C}$ and $L$, such that, for all $n \geq1$,
\[
\E_{\nu}( \MM_{n} \mid D_{n-1} < +\infty) \leq
c_{35}.
\]
\end{prop}

\begin{pf}
For $n=1$, the result is just (\ref{e:borne-esperance-0}). Consider $n
\geq1$, and write
%
\begin{equation}
\label{e:double-cond}\E_{\nu}( \MM_{n+1} \mid D_{n} <
+ \infty) = \frac{ \E_{\nu}( \MM_{n+1} \un(D_n<+\infty) \mid  D_{n-1}<
+\infty) }{\P_{\nu}( D_n < +\infty\mid  D_{n-1} < +\infty)}.
\end{equation}
Using Proposition~\ref{p:borne-esperance} and the fact that the event
$D_{n-1}<+\infty$ is measurable with respect to $\F^R_{S_n}$, we
deduce that
\[
\E_{\nu} \bigl( \MM_{n+1} \un(D_n<+\infty) \mid
D_{n-1}< +\infty \bigr) \leq c_{25} e^{-\theta L}
\E_{\nu} ( \MM_{n} \mid D_{n-1}<+\infty) +
c_{26}.
\]
On the other hand, observe that there exists a strictly positive
constant $d_1$ such that
%
\begin{equation}
\label{e:prob-echec} \P_{\nu}( D_n < +\infty\mid
D_{n-1} < +\infty) \geq d_1,
\end{equation}
considering, for example, the probability for the particle that makes
the front climb at time $S_{n}$ to cross at a time $> S_n$ the
half-line of slope $\alpha$ starting at $(S_n, r_{S_n})$.
Combining (\ref{e:double-cond}) and (\ref{e:prob-echec}), we deduce that
\[
\E_{\nu}( \MM_{n+1} \mid D_{n} < +\infty) \leq
d_1^{-1} c_{25} e^{-\theta L} \E_{\nu} (
\MM_{n} \mid D_{n-1}<+\infty) + d_1^{-1}
c_{26}.
\]
When $L$ is large enough so that $ d_1^{-1} c_{25}
e^{-\theta L} < 1$, we deduce, using also (\ref{e:borne-esperance-0}),
that the sequence
$ ( \E_{\nu}( \MM_{n} \mid  D_{n-1} < +\infty)  )_{n \geq
1}$ is bounded.
\end{pf}

We are now ready to prove our main estimate on the regeneration
structure, namely, Proposition~\ref{p:moments-renouvel}.

\begin{pf*}{Proof of Proposition~\ref{p:moments-renouvel}}
In this proof, we assume that $\mathscr{C}$ and $L$ are large enough so
that all the previous results hold. Remember the definition $\K= \inf
\{ n \geq1;   D_n = +\infty\}$. Our first claim is that, for some
strictly positive constant $d_1$ depending on $\mathscr{C}$
and $L$, for all $k \geq1$,
%
\begin{equation}
\label{e:succes} \P_{\nu}(\K\geq k) \leq d_1^k.
\end{equation}
From Corollary~\ref{c:prob-super-alpha}, we have that
\[
d_2:= \P_{\nu_{\mathscr{C},+}} \bigl( \{\mbox{$0$ is a forward super $
\alpha$-time for $B_0$} \} \cap G \bigr) > 0.
\]
Consider $n \geq1$. Using Proposition~\ref{p:conditionnement-S} and
the fact that
\[
G \subset G \bigl(Q^{(n)}, (r_{s+S_n} - r_{S_n})_{-S_n \leq s \leq0}, -r_{S_n} \bigr),
\]
we deduce that, on $\{ D_{n-1} < +\infty\}$, one has that
\[
\P_{\nu} \bigl(\mbox{$S_n$ is a forward and backward
super $\alpha$-time for $B_{S_n}$ } \mid \G^R_{S_n}
\bigr) \geq d_2 \qquad\mbox{a.s.}
\]
On the other hand, the event that $S_n$ is a forward sub $\alpha$-time
is measurable with respect to $\G^R_{S_n}$. Call $d_3$ the
probability for a random walk starting at zero to remain at zero during
the time-interval $[0,\alpha^{-1}]$ and then to satisfy $W_s \leq
\alpha s - 1$ for all $s \geq\alpha^{-1}$, which is $>0$ by Lemma
\ref
{l:atteinte-droite}. Using Lemma~\ref{l:controle-proba-atteinte}, we
deduce that, for arbitrary $K>0$, on $\{ D_{n-1} < +\infty\}$,
\[
\P_{\nu} \bigl(\mbox{$S_n$ is a forward sub $
\alpha$-time} \mid \F^R_{S_n} \bigr) \geq d_3 g(K)
\un( \MM_{n} \leq K) \qquad\mbox{a.s.}
\]
We deduce that
\[
\P_{\nu}(D_{n}<+\infty\mid D_{n-1}<+\infty) \geq
d_2 d_3 g(K) \P_{\nu}( \MM_{n} \leq K \mid
D_{n-1}<+\infty).
\]
By Proposition~\ref{p:bornitude}, we have that $\E_{\nu}( \MM_{n} \mid
D_{n-1} < +\infty) \leq c_{35}$, so that, by Markov's inequality,
$\P_{\nu}( \MM_{n} \leq2 c_{35} \mid  D_{n-1} < +\infty) \geq1/2$.
Setting $K:= 2 c_{35}$ and $d_1 :=1/2( d_2 d_3
g(K))$, we see that (\ref{e:succes}) is proved.

Now\vspace*{1pt} observe that, by definition, $S_{\K}$ is an $\alpha$-separation
time. As a consequence, we have that
$\kappa_1 \leq S_{\K}$.
Writing $S_{\K}:= S_1 + \sum_{k=1}^{\K-1} (S_{k+1}-S_k)$, we deduce
that for all $t$ and $n \geq1$,
%
\begin{eqnarray}
\label{e:decomp-telescopique}\P_{\nu}(\kappa_1 \geq t) &\leq&
\P_{\nu}( \K> n ) + \P_{\nu}(S_1 \geq t/n)
\nonumber\\[-8pt]\\[-8pt]\nonumber
&&{}+ \sum_{k=1}^{n-1} \P_{\nu}(
S_{k+1}-S_k \geq t/n, D_{k} < +\infty).
\end{eqnarray}
Let $t':=t/n$.
Using (\ref{e:queue-S'-S-0}) and (\ref{e:queue-S-S'-0}), we deduce that
%
\begin{equation}
\label{e:S1} \P_{\nu}(S_1 \geq t/n) \leq c_{30}
\bigl(t'/2 \bigr)^{- c_{31} \mathscr{C}} + c_{32} \bigl(t'/2
\bigr)^{- c_{33}
\mathscr{C}}.
\end{equation}
On the other hand, one has that
\begin{eqnarray*}
&& \P_{\nu}( S_{k+1}-S_k \geq t/n,
D_{k} < +\infty)
\\
&&\qquad \leq\P_{\nu}( S_{k+1}-S_k
\geq t/n, D_{k} < +\infty\mid D_{k-1}<+\infty),
\end{eqnarray*}
and, using (\ref{e:queue-S'-S}), (\ref{e:queue-S-S'}) and Proposition
\ref{p:bornitude}, we deduce that
%
\begin{eqnarray}
\label{e:Sk}
&& \P_{\nu}( S_{k+1}-S_k \geq t/n,D_{k} < +\infty)
\nonumber\\[-8pt]\\[-8pt]\nonumber
&&\qquad \leq  c_{35} e^{-\Clconstccseptyn (t'/2)} +
c_{11} \bigl(t'/2 \bigr)^{- c_{12} \mathscr{C}}
+ c_{16} \bigl(t'/2 \bigr)^{- c_{17}
\mathscr{C}}.
\end{eqnarray}
Choosing, for example, $n:= \lceil t^{1/2} \rceil $, and using (\ref
{e:succes}), (\ref{e:S1}) and (\ref{e:Sk}) to bound the terms in
(\ref{e:decomp-telescopique}), we deduce that
%
\begin{equation}
\label{e:queue-kappa} \P_{\nu}(\kappa_1 \geq t) \leq
d_4 t^{- d_5 \mathscr{C} + 1/2
},
\end{equation}
where $ d_4$ and $d_5$ are strictly
positive constants, with
$ d_4$ depending on $\mathscr{C}$ and $L$. Choosing
$\mathscr{C}$ large enough, this proves the fact that $\kappa_1$ has a
finite second moment. Now write
\[
\P_{\nu}(r_{\kappa_1} \geq\ell) \leq\P_{\nu}(
\kappa_1 > t) + \P_{\nu} \Bigl(\sup_{s \in[0,t]}
r_s \geq\ell \Bigr).
\]
Choosing $t:= \ell/ C'_1(\rho)$, and using Proposition~\ref
{p:ballistique-dessus-KS-extension} and
(\ref{e:queue-kappa}), we deduce that
%
\begin{equation}
\label{e:queue-rkappa} \P_{\nu}(r_{\kappa_1} \geq \ell) \leq
d_6 \ell^{- d_5 \mathscr{C} +
1/2},
\end{equation}
where $ d_6$ is a strictly positive constant depending
on $\mathscr{C}$ and $L$. Choosing $\mathscr{C}$ large enough, this
proves the fact that $r_{\kappa_1}$ has a finite second moment.
\end{pf*}

\section{Extension to the case $D_R>D_B$}\label{s:extension}

In this section, we briefly explain how to extend the approach leading
to Theorem~\ref{t:tcl} for the single-rate KS model, to the remanent KS
model with $D_R \geq D_B$, leading to Theorems~\ref{t:lgn-remanent} and
\ref{t:tcl-remanent}. The basic idea is to express the random walk
trajectories $(W,u)$ actually followed by particles in the model where
$D_R > D_B$, as
time-changed trajectories, with respect to random walk trajectories
$(\W,u)$ with constant jump rate. The key observation, stated as Lemma~\ref
{l:bleu-rouge} below, is that, due to the fact that $D_R > D_B$, the
set of particle trajectories $(W,u)$ that are blue at an upward time
for the front, coincides with the set of particle trajectories $(\W,u)$
whose position at time $t$ is above the position of the front (except
for the one that makes the front climb at time $t$). Finally, the
remanent character of the infection makes it possible to deduce
ballisticity estimates for the front by coupling with a single-rate model.

We now rigorously define the infection dynamics of the remanent
infection model for $D_R > D_B$, assuming without loss of generality
that $D_B=2$. To emphasize the similarities, we use as much as possible
the same notation that were already used for the single-rate KS
infection model.

We use a construction of the dynamics with $D_R > D_B=2$ that uses
random walk trajectories $(\W,u)$ with constant jump rate $2$, for
which our reference probability space for paths $(\W,u)$ is $(\Omega,
\F, \P_{w})$. As long as a particle is blue, it follows the corresponding
trajectory
in the usual way, while as soon as it is turned into a red particle, it
starts following the trajectory with a speed multiplied by a factor
$D_R/2$. As a result, the actual path $(W,u)$ followed by a particle is
related to the path $(\W,u) \in\Omega$ by a time-change, which we
describe below.

Let us first define the trajectory of the front. Since, by definition,
the front can only perform upward jumps, it makes sense to start with
$r_0:=0$, which leads to the simplification that $r_{T_k}:= k$ for all
$k \geq0$.
We start with $T_0:=0$, $r_0:=0$, and define inductively the sequence
$(T_{k})_{k \geq0}$ together with the
value of $(r_t)_{t \in[0,T_k]}$.
Consider $t>T_{\ell}$. We say that $t$ is upward if there exists
$(\W,u) \in\Psi$ such that $\W_s \leq r_s$ for some $s \in[0,t[$ and
such that $ \W_{v-} = \ell$ and $ \W_{v} = \ell+1$, where
%
\begin{equation}
\label{e:time-change}v:=\tau+\frac{D_R}{2}(t-\tau),\qquad \tau:= \inf \bigl\{ s \in[0,t[; \W_s \leq r_s \bigr\}.
\end{equation}
Then let
\[
T_{\ell+1}:= \inf\{ t > T_{\ell}; \mbox{ $t$ is upward} \},
\]
and
\[
r_t:= \ell\qquad\mbox{on }[T_{\ell}, T_{\ell+1}[.
\]
The sets $R_t$ and $B_t$ of red and blue particles at time $t$ are then
defined exactly as in the single-rate KS infection model, namely
\begin{eqnarray*}
B_t &:=& \bigl\{ (\W,u) \in\Psi; \forall s \in[0, t[,
\W_s > r_s \bigr\},
\\
R_t &:=& \bigl\{ (\W,u) \in\Psi; \exists s \in[0, t[,
\W_s \leq r_s \bigr\}.
\end{eqnarray*}

We now properly define $(W,u)$ as a time-changed version of $(\W,u)$.
Using the notation defined in (\ref{e:time-change}), we let $W_t:= \W
_t$ for $t \in[0,\tau]$ and $W_t:= \W_v$ for $t> \tau$. This construction is illustrated in Figure~\ref{f:simu-KS-couplage-XY-2-3}.

%
\begin{figure}

\includegraphics{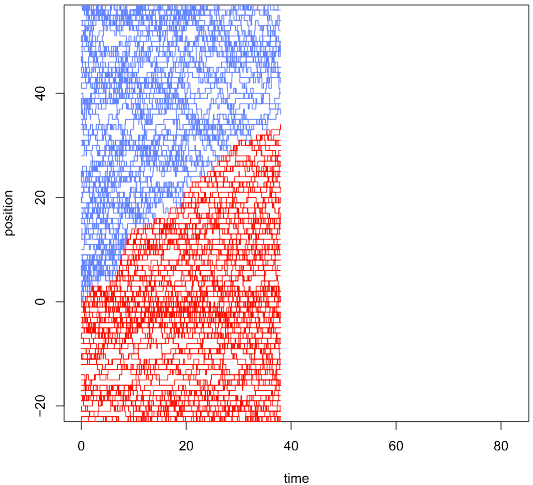}

\footnotesize{(a)~Actual trajectories}\vspace*{6pt}

\includegraphics{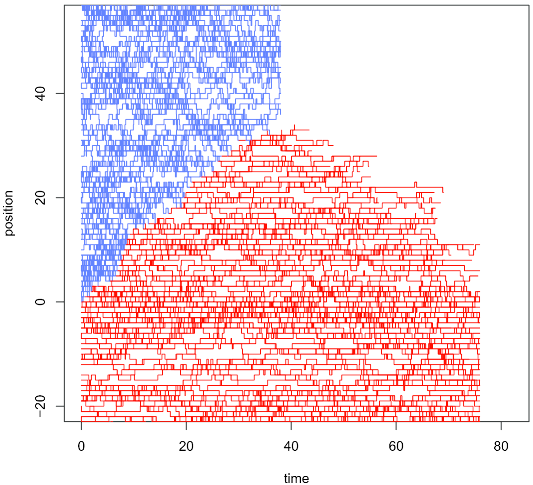}

\footnotesize{(b)~Time-changed trajectories with constant jump rate $D_B=1$}

\caption{Realization of the KS infection model with $D_R=2$ and
$D_B=1$. The actual evolution of the process is shown in~\textup{(a)}. The evolution of the corresponding
time-changed trajectories with a constant jump rate $D_B=1$ is shown in~\textup{(b)}.}
\label{f:simu-KS-couplage-XY-2-3}
\end{figure}

We now make the following key remark.
%
\begin{lemma}\label{l:bleu-rouge}
For all $k \geq1$, the set $B_{T_k}$ coincides with the set of $(\W,u)
\in\Psi$ such that $\W_{T_k} \geq k$, minus the particle that makes
the front climb at time $T_k$.
\end{lemma}
Note that the above result is an immediate consequence of the
definition when $D_R=D_B$, but not in the present case, due to the time-change.
\begin{pf}
One inclusion is immediate: a particle path $(\W,u)$ in $B_{T_k}$
evolves using the jump rate $D_B=2$ up to at least time $T_k$, so that
$\W_{T_k}$ indeed corresponds to
the position $W_{T_k}$ of the corresponding particle at time $T_k$, and
must by definition be $\geq k$. On the other hand, assume that a $(\W,u) \in R_{T_k}$ is such that
$\W_{T_k} \geq k$, and hits (or lies below, to include particles in
$R_{0+}$) the front for the first time at a time $\tau<T_k$. Introduce
the time $t:= \tau+(T_k-\tau)\frac{2}{D_R}$.
Since\vspace*{2pt} $D_R>D_B=2$, we have $t < T_k$, and by definition one has $W_t =
\W_{T_k} \geq k$, whence the existence of a red particle above $k$ at a
time $<T_k$, which contradicts the definition of $T_k$.
\end{pf}

One now defines the renewal structure exactly as for the single-rate KS
infection model, but with the time-changed trajectories $W$ replacing
the trajectories $\W$.
Similarly, we can define
\begin{eqnarray*}
\F^R_t &:=& \sigma \bigl( (W_s,u); s \leq t,
(\W,u) \in R_t \bigr),
\\
\F^R_T &:=& \sigma(T, r_T) \vee\sigma \bigl(
(W_s,u); s \leq t, (\W,u) \in R_T \bigr),
\\
\G^R_t&:=& \sigma \bigl( (W_s,u); s \in\R,
(\W,u) \in R_t \bigr),
\\
\G^R_T&:=& \sigma(T, r_T) \vee\sigma \bigl(
(W_s,u); s \in\R, (\W,u) \in R_T \bigr).
\end{eqnarray*}
Note that it does not matter whether we define the $\sigma$-algebras
$\G
^R_{t}$ using the original or time-changed trajectories, since in both
cases the history of the front up to time $t$ is measurable, due to the
fact that the $\sigma$-algebra includes the full trajectories (and not
just the trajectories up to time $t$). The same remark is valid for $\G
^R_T$, where $T$ is a nonnegative random time.
With the help of Lemma~\ref{l:bleu-rouge}, and of the fact that, for
any $(\W,u) \in B_{T_k}$, one has $W_s=\W_s$ for all $s \leq T_k$,
it is then possible to re-prove Propositions~\ref{p:renouvel-1} in
exactly the same way as for the single-rate KS infection model.

The key advantage of introducing remanence in the model is that, when
$D_R > D_B=2$, a comparison holds with the single rate model with jump
rate equal to $2$.
%
\begin{lemma}\label{l:comparaison-remanent}
Let $\rtun$ denote the front of the single-rate KS model with rate $2$,
and $\rtdeux$ denote the front of the remanent KS model. If $D_R >
D_B=2$, one has that
$\rtun\leq\rtdeux$ for any $t$.
\end{lemma}
The above lemma, combined with Proposition~\ref{p:ballistique-dessous},
yields the key ballisticity estimate needed to reprove the estimates of
Section~\ref{s:estimates} for the remanent KS infection model.
The two additional results we need are the following: a version of the
strong Markov property restricted to $R_T$, and an upper bound on the
speed exactly similar to Proposition~\ref{p:ballistique-dessus-KS}.
Specifically, we have the following.
%
\begin{prop}\label{p:strong-Markov-time-change}
The strong Markov property holds for our process: for all $w \in\S
_{\theta}$, all nonnegative $(\F^R_t)_{t \geq0}$-stopping time $T$,
and bounded measurable function $F$ on $\DD_+$,
one has that, on $\{ T < +\infty\}$,
%
\begin{equation}
\label{e:strong-Markov-time-change}\E_w \bigl( F \bigl( X(R_T) \bigr) \mid
\F^R_T \bigr) = \E_{X_T(R_T)} \bigl(F(X) \bigr)\qquad
\P_w\mbox{-a.s.},
\end{equation}
where we use the notation $X:=(X_{t})_{t \geq0}$.
\end{prop}
%
%
\begin{prop}\label{p:ballistique-dessus-remanent-KS}
For the remanent KS infection model, there exist a constant $C^{\star
}_1(\rho)>0$ and a constant $c_{36}$, depending on $\rho$ and
$\mathscr{C}$, such that, for every $t > 0$,
\[
\P_{\nu_{\mathscr{C},+}} \bigl(r_t \geq C^{\star}_1(
\rho) t \bigr) \leq c_{36} \exp(-t).
\]
\end{prop}
It is then possible to reprove all the estimates of Section~\ref
{s:estimates}, the only difference being that, at some places,
estimates for a random walk with constant jump rate 2 have to be
replaced by estimates
for a random walk whose jump rate may change from $D_B=2$ to $D_R>2$ at
some time-point. These estimates are obtained by a simple comparison
with a random walk with constant jump rate equal to $D_R$. One then
obtains Proposition~\ref{p:moments-renouvel}, leading to the proof of
the law of large numbers (Theorem~\ref{t:lgn-remanent}), and the
central limit theorem (Theorem~\ref{t:tcl-remanent}).





\printaddresses
\end{document}